\documentclass[11pt,leqno]{article}

\usepackage{amsmath, amsfonts, theorem,color,
}
\usepackage[latin1]{inputenc}

\usepackage[T1]{fontenc}

\usepackage[english]{babel}

\usepackage{lmodern}

\usepackage{amsmath}

\usepackage{amssymb}

\usepackage{appendix}

\usepackage{mathrsfs}
\usepackage{latexsym}
\usepackage{graphicx}
\usepackage{cancel}
\usepackage[normalem]{ulem}
\makeatletter
\def\@maketitle{\newpage
    \null
    \vskip .8truein
    \begin{center}%
     {\bf \@title \par}%
     \vskip 1.5em
     {\small
      \lineskip .5em
      \begin{tabular}[t]{c}\@author
      \end{tabular}\par}%
    \end{center}%
    \par
    \vskip .4truein}
\@addtoreset{equation}{section} \@addtoreset{theorem}{section}
\@addtoreset{lemma}{section} \@addtoreset{proposition}{section}
\@addtoreset{definition}{section} \@addtoreset{corollary}{section}
\@addtoreset{remark}{section}

\def\dfrac#1#2{\ds{\frac{#1}{#2}}}

\let\nn=\nonumber

\newcommand{\re}{{\mathbb R}}

\newcommand{\He}{{\mathbb H}}
\newcommand{\cH}{{\mathcal H}}

\let\ds=\displaystyle

\def\R{{\mathbb R}}

\let\a=\alpha

\newtheorem{theorem}{Theorem}[section]
\newtheorem{lemma}{Lemma}[section]
\newtheorem{proposition}{Proposition}[section]
\newtheorem{definition}{Definition}[section]
\newtheorem{corollary}{Corollary}[section]
\newtheorem{remark}{Remark}[section]
\newtheorem{example}{Example}[section]
\newtheorem{hypothesis}{Hypothesis}[section]

  {\hfill$\Box$\bigskip\par}

\DeclareMathOperator{\diver}{div}

\def\proof{\list{}{\setlength{\leftmargin}{0pt}
                      \parskip=0pt\parsep=0pt\listparindent=2em
                      \itemindent=0pt}\item[]\futurelet\testchar\@maybe}

\def\@maybe{\ifx[\testchar \let\next\@Opt
          \else \let\next\@NoOpt \fi \next}
\def\@Opt[#1]{{\it Proof of #1.\ }}\def\@NoOpt{{\it Proof.\ }}
\begin{document}
\title{\Large \bf Non coercive unbounded first order Mean Field Games: the Heisenberg example }

\author{{\large \sc Paola Mannucci\thanks{Dipartimento di Matematica ``Tullio Levi-Civita'', Universit\`a di Padova, mannucci@math.unipd.it}, Claudio Marchi \thanks{Dipartimento di Ingegneria dell'Informazione \& Dipartimento di Matematica ``Tullio Levi-Civita'', Universit\`a di Padova, claudio.marchi@unipd.it}, Nicoletta Tchou \thanks{Univ Rennes, CNRS, IRMAR - UMR 6625, F-35000 Rennes, France, nicoletta.tchou@univ-rennes1.fr}}} 
\maketitle

\begin{abstract}
In this paper we study evolutive first order Mean Field Games in the Heisenberg group;
%~$\He^1$; 
each agent can move in the whole space but it has to follow ``horizontal'' trajectories which are given in terms of the vector fields generating the group and the kinetic part of the cost depends only on the horizontal velocity. The Hamiltonian is not coercive in the gradient term and the coefficients of the first order term in the continuity equation may have a quadratic growth at infinity.
The main results of this paper are two: the former is to establish the existence of a weak solution to the Mean Field Game systems while the latter is to represent this solution following the Lagrangian formulation of the Mean Field Games.
We also provide some generalizations to Hei\-sen\-berg-type structures.
\end{abstract}



\section{Introduction}
In this paper we study evolutive first order Mean Field Game (briefly, MFG) systems in the Heisenberg group~$\He^1$. Let us recall that the MFG theory started with the works by Lasry and Lions \cite{LL1,LL2,LL3} and by Huang, Malham\'e and Caines \cite{HMC} (see the works \cite{C, BFY, GPV} for the many developments in recent years) and studies Nash equilibria when the number of agents tends to infinity and each agent's aim is to control its dynamics so to minimize a given cost which depends on the distribution of the whole population. On the other hand, the Heisenberg group can be seen as the first non-Euclidean space which is still endowed with nice properties as a (noncommutative) group operation, a family of dilations and a manifold structure (see the monographs \cite{BLU, Montg} for an overview). The Heisenberg setting gives rise to non holonomic constraints in the optimal problem addressed by the agent, see
\cite[p.135]{Cor} and \cite[p.52]{Bel}. Moreover let us recall that many control problems, as the vehicle model, can be written in a form similar to the Heisenberg one: for real models and numerical analysis see \cite{AC08, MuSa, DFPD} and references therein.
From the viewpoint of a single agent in the MFG, the Heisenberg's framework entails that its state cannot change isotropically in all the directions but it can move only along {\it admissible} trajectories.

We shall consider systems of the form
\begin{equation}\label{eq:MFGintrin}
\left\{\begin{array}{lll}
(i)&\quad-\partial_t u+\frac{|D_{\cH}u|^2}{2}=F[m(t)](x)&\qquad \textrm{in }\He^1\times (0,T)\\
(ii)&\quad\partial_t m-\diver_{\cH}  (m D_{\cH}u)=0&\qquad \textrm{in }\He^1\times (0,T)\\
(iii)&\quad m(x,0)=m_0(x), u(x,T)=G[m(T)](x)&\qquad \textrm{on }\He^1,
\end{array}\right.
\end{equation}
where $D_{\cH}$ and $\diver_{\cH}$ are respectively the {\it horizontal gradient} and the {\it horizontal divergence} while $F$ and $G$ are strongly regularizing coupling operators.
For readers which are not familiar with intrinsic calculus, in Euclidean coordinates, system~\eqref{eq:MFGintrin} becomes
\begin{equation}\label{eq:MFG1}
\left\{\begin{array}{lll}
(i)&\quad-\partial_t u+H(x, Du)=F[m(t)](x)&\qquad \textrm{in }\re^3\times (0,T)\\
(ii)&\quad\partial_t m-\diver  (m\, \partial_pH(x, Du))=0&\qquad \textrm{in }\re^3\times (0,T)\\
(iii)&\quad m(x,0)=m_0(x), u(x,T)=G[m(T)](x)&\qquad \textrm{on }\re^3,
\end{array}\right.
\end{equation}
where, for $p=(p_1,p_2, p_3)$ and $x=(x_1, x_2,x_3)$, the Hamiltonian $H(x,p)$ is
\begin{equation}\label{Hd}
H(x,p):=\frac{1}{2}((p_1-x_2p_3)^2+(p_2+x_1p_3)^2)=\frac{|pB(x)|^2}{2}\ \textrm{with }\
 B(x):=  \begin{pmatrix}
\!\!&1& 0\!\\
\!\!&0&1&\!\\
\!\!&-x_2&x_1\!
\end{pmatrix}\in M^{3\times 2}
\end{equation}
while the drift $\partial_pH(x,p)$ is
\begin{equation}\label{Hp}
\partial_pH(x,p)=pB(x)B(x)^T=(p_1-x_2p_3, p_2+x_1p_3, -p_1x_2+p_2x_1+p_3(x_1^2+x_2^2)).
\end{equation}

These MFG systems arise when the generic player with state~$x$ at time~$t$ can move in the whole space but it must follow {\it horizontal curves} with respect to the two vector fields $X_1$ and $X_2$ generating the Heisenberg group (see \eqref{vectorFields} below):
\begin{equation}\label{DYNH}
x'(s)=\alpha_1(s)X_1(x(s))+\alpha_2(s)X_2(x(s))
\end{equation}
namely
\begin{equation*}
x_1'(s)=\alpha_1(s),\quad x_2'(s)=\alpha_2(s),\quad x_3'(s)=-x_2(s)\alpha_1(s)+x_1(s)\alpha_2(s).
\end{equation*}
Each agent wants to choose the control $\alpha=(\alpha_1, \alpha_2)$ in $L^2([t,T];\R^2)$ in order to minimize the cost
\begin{equation}\label{Jgen}
J_{x,t}^{m}(\alpha):=\int_t^T\left[\frac12 |\alpha(\tau)|^2+F[m(\tau)](x(\tau))\right]\,d\tau+G[m(T)](x(T))
\end{equation}
where $m(\cdot)$ is the evolution of the whole population's distribution while $(x(\cdot),\alpha(\cdot))$ is a trajectory obeying to~\eqref{DYNH}.

Let us observe three important issues of these MFG systems: $(i)$ the Hamiltonian~$H$ is not coercive in~$p$, $(ii)$ the system is in the whole space, $(iii)$ in equation~\eqref{eq:MFG1}-(ii) the coefficient of the first order term may have quadratic growth in~$x$.\\
Point~$(i)$ prevents the application of standard approaches for first order MFG (for instance, see~\cite{BFY,C,CGPT}) because they require uniform coercivity of the Hamiltonian. Moreover, we recall that the papers~\cite{AMMT, CM, MMMT} already tackled MFG systems with noncoercive Hamiltonians for first order MFG while papers \cite{DF, FGT} dealt with second order hypoelliptic MFG. However, the results in~\cite{AMMT,CM} do not apply to the present setting because these papers consider a different kind of admissible trajectories.
Note that the present case is not even encompassed in the previous work~\cite{MMMT} because $\det B(x)B(x)^T= 0$ for any $x\in \re^3$.
The degeneracy of the matrix $B(x)B(x)^T$ implies that we cannot prove the uniqueness of optimal trajectories for a.e. starting points with respect to the initial distribution of players and hence to get a representation formula as in \cite{C, MMMT}.
The issue of finding necessary or sufficient conditions ensuring the uniqueness of the optimal trajectories for a.e. starting points is challenging and open; we hope to study it in a future work.\\
On the other hand, points~$(ii)$ and~$(iii)$ give rise to some difficulties for applying the vanishing viscosity method, especially for the Cauchy problem for equation~\eqref{eq:MFG1}-(ii) with the viscosity term. Actually in this problem the coefficients grow ``too much at infinity'' and one cannot invoke nor standard results for the well-posedness of the problem neither its interpretation in terms of a stochastic optimal control problem. 

The aims of this paper are two; the former one is to prove the existence of a weak solution to system~\eqref{eq:MFGintrin} while the latter, and main, one is to prove that this weak solution is also a {\it mild} solution in the sense introduced by Cannarsa and Capuani~\cite{CC} for the case of state-constrained MFG where the agents control their velocity.\\
In order to obtain the existence of a weak solution, we establish several properties of the solution to the Hamilton-Jacobi equation~\eqref{eq:MFGintrin}-(i) (as semiconcavity, Lipschitz continuity, regularity of the optimal trajectories for the associated optimal control problem). Afterwards, we adapt the techniques introduced by PL. Lions in his lectures at Coll\`ege de France \cite{C, LL3} (see also \cite{AMMT, MMMT} for similar approaches for some noncoercive Hamiltonians).
To get the result we perform three approximations: a completion $B^\varepsilon$ of~$B$ (see~\eqref{He}), 
a vanishing viscosity procedure with the {\it Euclidean} Laplacian and a truncation argument of the coefficients of matrix~$B$. 
The completion $B^\varepsilon$ fulfills $\det B^\varepsilon(x)(B^\varepsilon(x))^T\ne 0$ for any $x\in \re^3$ which is a crucial property for getting uniqueness of optimal trajectory for $m_0$-a.e. starting point.
The vanishing viscosity procedure permits to exploit the regularity results of the Laplacian while the truncation argument permits to avoid parabolic Cauchy problems with coefficients growing ``too much'' at infinity.  \\
Let us note that the matrix~$B^\varepsilon$ is associated to different constraints on horizontal curves (see \eqref{DYNe} below) and that the geometry of the space changes drastically as $\varepsilon\to 0^+$; we refer the reader to the paper~\cite{CCR} for a discussion on this issue. 
Finally, we shall prove that this weak solution is also a {\it mild} solution in the sense introduced in~\cite{CC}.
Roughly speaking, as in the Lagrangian approach for MFG (see~\cite{BCS,CC}), this property means that, for a.e. starting state, the agents follow optimal trajectories for the optimal control problem associated to the Hamilton-Jacobi equation.
In order to prove that our solution is in fact a mild solution, we shall use the superposition principle~\cite[Theorem 8.2.1]{AGS}.
To do this we need to prove in the Heisenberg framework an optimal synthesis result and the aforementioned properties of the matrix~$B^\varepsilon$.

It is worth noting that our techniques relies on some compactness of initial distribution of players and on sublinear growth of the coefficients of~$B$ but they do not need the H\"ormander condition. Indeed, we present our results for purely quadratic Hamiltonian  as in~\eqref{Hd} on the (first) Heisenberg group~$\He^1$ only for the sake of simplicity. As a matter of facts, 
our results can be extended to any Hamiltonian of the form $H(x,p)=|pB(x)|^\gamma$ with $\gamma\in[1,2]$ and some structures of {\it Heisenberg type} (see \cite[Chapter 18]{BLU} for precise definition and main properties). 
As an example, we apply our result to the case of classical Grushin dynamics with unbounded coefficients. For the case of Grushin dynamics with bounded coefficients, we refer the reader to \cite[Theorem 1.1]{MMMT} where a more precise interpretation of the system is obtained taking advantage of such a boundedness.

This paper is organized as follows. Section \ref{Preliminaries} is devoted to give the main definitions including the Heisenberg group, contains the assumptions and the statement of our main result (Theorem~\ref{thm:main_illi}) whose proof is postponed in Section \ref{sect:pr_illi}. 
In  Section \ref{OC} we study several properties of the solution of the optimal control problem associated to the Hamilton-Jacobi equation \eqref{eq:MFGintrin}-(i). Section~\ref{sect:c_illi} is devoted to establish the well posedness of the continuity equation; in particular, Theorem~\ref{prp:m} states some regularity estimates which are crucial in the proof of Theorem~\ref{thm:main_illi}.
In Section \ref{ext} we provide two generalizations of our result: in the former we consider structures of Heisenberg type while in the latter
we tackle power-type Hamiltonians with exponent $\gamma\in[1,2]$.

%%%%%%%%%%%%%%%%%

\section{Preliminaries: definitions, assumptions and main results}\label{Preliminaries}
In this section, we introduce the notations (including the functional spaces needed for the definition of solution to system~\eqref{eq:MFGintrin} and the Heisenberg group), fix our assumptions and state the main results of this paper.

\subsection{Notations and Heisenberg group}
For any function $u:\re^n\times\re\ni (x,t)\to u(x,t)\in \re$, $Du$ and~$D^2u$ stand for the Euclidean gradient and respectively Hessian matrix with respect to~$x$.
We denote $C^2(\re^n)$ the space of functions with continuous second order derivatives  and we write $\|f\|_{C^2(\re^n)}:=\sup_{x\in \re^n}[|f(x)|+|Df(x)|+|D^2f(x)|]$.\\
For any complete separable metric space $X$, ${\mathcal P}(X)$ denotes the set of Borel probability measures on~$X$. For any complete separable metric spaces $X_1$ and $X_2$, any measure $\eta\in{\mathcal P}(X_1)$ and any function $\phi:X_1\to X_2$, we denote $\phi\#\eta\in{\mathcal P}(X_2)$ the {\it push-forward} of~$\eta$ through~$\phi$, i.e. $\phi\#\eta(B):=\eta(\phi^{-1}(B))$, for any $B$ measurable set, $B\subset X_2$ (see~\cite[section~5.2]{AGS} for its main properties).
For a function $m\in C([0,T],{\mathcal P}(X))$, $m_t$ stands for the probability~$m(t,\cdot)$ on~$X$.\\
We introduce the functional spaces ~$\mathcal P_1(\re^n)$ (respectively, $\mathcal P_2(\re^n)$) as the space of Borel probability measures on~$\re^n$ with finite first (resp., second) order moment with respect to the Euclidean distance, endowed with the Monge-Kantorovich distance~${\bf d_{1}}$ (resp.,~${\bf d_{2}}$) (for more details see \cite{AGS} or \cite{C}).
\vskip .2truecm
\noindent We refer to \cite{BLU} for a complete overview on the Heisenberg group~$\He^1$.
We define the two vector fields, called {\it generators} of $\He^1$,
\begin{equation}\label{vectorFields}
 X_1(x):=\left(\begin{array}{c}1 \\0 \\
 -x_2\end{array}\right)\quad
\textrm{and}
\quad
X_2(x):=\left(\begin{array}{c}0 \\1 \\x_1\end{array}\right),
\quad \forall\, x=(x_1,x_2,x_3)\in \re^3.
\end{equation}
By these vectors we define the linear differential operators, still called $X_1$ and $X_2$
\begin{equation}\label{VFD}
X_1=\partial_{x_1}-x_2\partial_{x_3},\qquad X_2=\partial_{x_2}+x_1\partial_{x_3}.
\end{equation}
Note that their {\it commutator} $[X_1,X_2]:=X_1\,X_2-X_2X_1$ verifies: $[X_1,X_2]=-2\partial_{x_3}$; hence, $X_1$, $X_2$ and $[X_1,X_2]$ span all $\re^3$.

The Heisenberg group~$\He^1$ has a group structure endowed with the following noncommutative group operation, denoted by $\oplus$: for all $x=({x}_1,{x}_2,x_3)$, $y=({y}_1,{y}_2,y_3)\in \R^{3}$,
\begin{equation*}%\label{Group_Law}
x\oplus y=(x_1, x_2, x_3)\oplus (y_1, y_2, y_3):= (x_1+y_1, x_2+y_2, x_3+y_3-x_2y_1+x_1y_2).
\end{equation*}

The fields $X_1$ and $X_2$ are left-invariant vector fields, i.e.
for all $u\in C^{\infty}(\R^{3})$ and for all fixed $y\in \R^{3}$
we have
$X_i(u(y\oplus x))=(X_iu)\,(y\oplus x),\ i=1,2.
$\\
Note that the matrix $B(x)$ defined in \eqref{Hd} is the matrix associated to the vectors $X_1$ and $X_2$. For any regular real-valued function $u$, we shall denote its horizontal gradient and its horizontal Laplacian by $D_\cH u:= (X_1u, X_2u)$ and respectively $\Delta_\cH:=X_1^2u+X_2^2u$ and we observe $D_\cH u=Du\, B(x)$ and $\Delta_\cH u=\textrm{tr}(D^2u\,BB^T)$. 
For any regular $v=(v_1,v_2):\R^{3}\to \re^2$, we denote its horizontal divergence by $\diver_{\cH}v:=X_1v_1+X_2v_2$ and we note that the left-invariance of $X_i$ ($i=1,2$) entails the left-invariance of $\diver_{\cH}$. We have: $\diver_{\cH}(D_\cH u)=\Delta_\cH u$.

The norm and the distance associated by the group law are defined as
\begin{equation*}\label{norm}
\|x\|_\cH:=((x_1^2+x_2^2)^2+x_3^2)^{1/4}, \quad d_\cH(x,y):=\|x\oplus y^{-1}\|_\cH.
\end{equation*}

\subsection{Assumptions, definitions of solution and main result}
Throughout this paper (unless otherwise explicitly stated) we shall require the following hypotheses
\begin{enumerate}
\item[(H1)]\label{H1'} the functions~$F$ and $G$ are real-valued function, continuous on ${\mathcal P}_{1}(\re^3)\times\re^3$,
\item[(H2)]\label{H2'} the map $m\to F[m](\cdot)$ is Lipschitz continuous from~${\mathcal P}_{1}(\re^3)$ to $C^{2}(\re^3)$; moreover, there exists~$C\in \mathbb R$ such that
  $$\|F[m](\cdot)\|_{C^{2}(\re^3)}, \|G[m](\cdot)\|_{C^{2}(\re^3)}\leq C,\qquad \forall m\in {\mathcal P}_{1}(\re^3);$$
\item[(H3)]\label{H4'} the distribution~$m_0:\re^3\to \re$ is a nonnegative $C^{0}$ function with compact support and $\int_{\re^3}m_0dx=1$.
\end{enumerate}
\vskip .2truecm
We now introduce our definitions of weak solution of the MFG system~\eqref{eq:MFGintrin}.
\begin{definition}\label{defsolmfg}
A couple $(u,m)$ of functions defined on~$\re^3\times [0,T]$ is a weak solution of system~\eqref{eq:MFGintrin} if
\begin{itemize}
\item[1)] $u$ belongs to $W^{1,\infty}(\re^3\times[0,T])$;
\item[2)] $m$ belongs to $C^0([0,T];{\mathcal P}_{1}(\re^3))$ and for all $ t\in [0,T]$, $m(t)$ is absolutely continuous w.r.t. the Lebesgue measure. Let $m(\cdot, t)$ denote the density of $m(t)$. The function $(x,t)\mapsto m(x,t)$ is bounded;
\item[3)] Equation~\eqref{eq:MFGintrin}-(i) is satisfied by $u$ in the viscosity sense in~$\re^3\times(0,T)$;
\item[4)] Equation~\eqref{eq:MFGintrin}-(ii) is satisfied by $m$ in the sense of distributions  in~$\re^3\times(0,T)$.
\end{itemize}
\end{definition}

\begin{remark}
From \cite[Lemma 8.1.2]{AGS}, we get that the distributional solution of \eqref{eq:MFGintrin}-(ii) stated in point 4) of the definition \ref{defsolmfg}
is automatically continuous in the sense of point 2) of the same definition.
\end{remark}
We introduce now the notion of {\it mild} solution introduced by~\cite{CC}.
This notion is reminiscent of the Lagrangian approach to MFGs (see~\cite{BCS}) and it relies on replacing probability measures on the state space with probability measures on arcs on the state space.\\
We define the set of AC arcs in $\re^3$
\begin{equation}\label{gamma}
\Gamma:=AC((0,T), \re^3)
\end{equation}
and the evaluation map $e_t: \Gamma\to \re^3$ as
\begin{equation}\label{eval}
e_t(\gamma)=\gamma(t).
\end{equation}
For any $x\in \re^3$, we define the set of horizontal arcs starting at~$x$ with an associated control law
\begin{equation*}
\mathcal A(x,t):=\{(\gamma, \alpha) : \gamma\in \Gamma,\ \gamma(t)=x,\  \alpha\in L^2([t,T], \re^2), (\gamma, \alpha) \textrm{ solves \eqref{DYNH} in $(t,T)$}\}.
\end{equation*}
Given $m_0\in{\mathcal P}_{1}(\re^3)$,  we define
$$
\mathcal P_{m_0}(\Gamma)=\{\eta\in{\mathcal P}(\Gamma):\ m_0 =e_0\# \eta\}.
$$
For any $\eta\in \mathcal P_{m_0}(\Gamma)$, $t\in[0,T]$ and $x\in \re^3$, we consider the cost
\begin{equation}\label{Jeta}
J^{\eta}_{x,t}(\alpha):=\int_t^T\left[\frac12 |\alpha(\tau)|^2+F[e_\tau\#\eta](\gamma(\tau))\right]\,d\tau+G[e_T\#\eta](\gamma(T))
\end{equation}
where $(\gamma, \alpha) \in \mathcal A(x,t)$.
For any $\eta\in \mathcal P_{m_0}(\Gamma)$ and for any $x\in \re^3$ we define the set of optimal horizontal arcs starting at~$x$
\begin{equation}\label{Gammaeta}
\Gamma^{\eta}[x]:= \{\overline\gamma:\ (\overline\gamma, \overline\alpha) \in \mathcal A(x,0):
J^{\eta}_{x,0}(\overline\alpha)= \min_{(\gamma, \alpha) \in \mathcal A(x,0)} J^{\eta}_{x,0}(\alpha)\}.
\end{equation}
\begin{definition}\label{mfgequil}
A measure $\eta\in \mathcal P_{m_0}(\Gamma)$ is a {\em MFG equilibrium} for $m_0$ if
$$supp\, \eta\subseteq \bigcup_{x\in \re^3} \Gamma^{\eta}[x].$$
\end{definition}
This means that the support of $\eta$ is contained in the set 
$$\cup_{x\in\re^3}\{\gamma\in \Gamma: \gamma(0)=x,\, \textrm{$\gamma$ is a minimizer of $J^{\eta}_{x,0}$}\}.$$

\begin{definition}\label{mild}
A couple $(u,m)\in C^0([0,T]\times \re^3)\times C^0([0,T]; \mathcal P_{1}(\re^3))$ is called {\em {mild solution}} if there exists a MFG equilibrium $\eta$ for $m_0$ such that:
\begin{itemize}
\item[i)]
$m_t =e_t\# \eta$;
\item[ii)] $u$ is given by
$$
u(x,t)= \inf_{(\gamma, \alpha) \in \mathcal A(x,t)}\int_t^T\left[\frac12 |\alpha(\tau)|^2+F[e_\tau\#\eta](\gamma(\tau))\right]\,d\tau+G[e_T\#\eta](\gamma(T))
.$$
\end{itemize}
\end{definition}
\vskip .2truecm
We now state the main result of this paper.
\begin{theorem}\label{thm:main_illi}
Under the above assumptions:
\begin{itemize}
\item[i)] System \eqref{eq:MFGintrin} has a solution $(u,m)$;
\item[ii)] $(u,m)$ is a mild solution.
\end{itemize}
\end{theorem}

\begin{remark}
\begin{itemize}
\item [-]As a matter of fact, in the proof of this theorem we get that any solution in the sense of Definition \ref{defsolmfg} is a mild solution.
\item [-]Uniqueness holds under classical hypothesis on the monotonicity of $F$ and $G$ as in \cite{C}.
\end{itemize}
\end{remark}

\subsection{The $\epsilon$-approximating problem}
In order to prove Theorem~\ref{thm:main_illi}, it is expedient to introduce the following approximating problems for $\varepsilon\in(0,1]$

\begin{equation}\label{eq:MFGe}
\left\{\begin{array}{lll}
(i)&\quad-\partial_t u+H^{\varepsilon}(x, Du)=F[m(t)](x)&\qquad \textrm{in }\re^3\times (0,T)\\
(ii)&\quad\partial_t m-\diver  (m\, \partial_pH^{\varepsilon}(x, Du))=0&\qquad \textrm{in }\re^3\times (0,T)\\
(iii)&\quad m(x,0)=m_0(x), u(x,T)=G[m(T)](x)&\qquad \textrm{on }\re^3,
\end{array}\right.
\end{equation}
where
\begin{equation}\label{He}
H^{\varepsilon}(x,p)=\frac{1}{2}|pB^{\varepsilon}(x)|^2\qquad\textrm{with}\qquad
 B^\varepsilon:=  \begin{pmatrix}
      \!\!&1& 0&0&\!\\
     \!\!&0&1&0&\!\\
      \!\!&-x_2&x_1&\epsilon&
\end{pmatrix}.
    \end{equation}
Hence, explicitly, the Hamiltonian and the drift are respectively
\begin{eqnarray*}
H^{\varepsilon}(x,p)&=&\frac{1}{2}((p_1-x_2p_3)^2+(p_2+x_1p_3)^2+(\epsilon p_3)^2)\\
\partial_pH^{\varepsilon}(x,p)&=& p B^{\varepsilon}(x)(B^{\varepsilon}(x))^T=(p_1-x_2p_3, p_2+x_1p_3,x_1p_2-x_2p_1+(x_1^2+x_2^2+\varepsilon^2)p_3)
\end{eqnarray*}
while the dynamics of the generic player at point~$x$ at time~$t$ becomes
\begin{equation}\label{DYNe}
x_1'(s)=\alpha_1(s),\quad x_2'(s)=\alpha_2(s),\quad x_3'(s)=-x_2(s)\alpha_1(s)+x_1(s)\alpha_2(s)+\epsilon \alpha_3(s)
\end{equation}
where the control $\alpha=(\alpha_1, \alpha_2, \alpha_3)$ is chosen in $L^2([t,T];\R^3)$ for minimizing the cost~\eqref{Jgen}.

We can obtain existence, representation formula and suitable estimates of a solution to problem~\eqref{eq:MFGe} which will allow us to prove Theorem \ref{thm:main_illi} letting $\varepsilon\to0^+$. 
These properties are stated in the following Proposition whose proof, together with the proof of Theorem \ref{thm:main_illi},
is postponed in section~\ref{sect:pr_illi}.
\begin{proposition}\label{mainproe}
For any fixed $\epsilon\in(0,1]$ there exists a solution $(u_{\varepsilon}, m_{\varepsilon})$ of the system
\eqref{eq:MFGe} such that
\begin{equation}
\label{ambrosioe}
\int_{\re^3}\phi(x) \, dm_\varepsilon(t)=\int_{\re^3}\phi(\gamma_{\varepsilon,x}(t))\,m_0(x)\, dx \qquad \forall \phi\in C^0_0(\R^3), \, \forall t\in[0,T]
\end{equation}
where, for a.e. $x\in\re^3$,  $\gamma_{\varepsilon,x}$ is the unique solution to
\begin{equation}\label{flowe}
x'(s)= -Du_{\varepsilon}(x(s),s)B^{\epsilon}(x(s))(B^{\epsilon}(x(s)))^T,\quad x(0)=x.
\end{equation}
Moreover, there exists a positive constant~$C$ (independent of~$\epsilon$) such that
\begin{itemize}
\item[a)] $\|u_\varepsilon\|_\infty\leq C$, $\|D u_\varepsilon\|_\infty\leq C$, $|\partial_t u_\varepsilon(t,x)|\leq C (1+|x_1|^2+|x_2|^2)$, $D^2 u_\varepsilon\leq C$,
\item[b)] $\|m_\varepsilon\|_\infty\leq C$, ${\bf d}_1(m_\varepsilon(t_1),m_\varepsilon(t_2))\leq C|t_2-t_1|^{1/2}$, $\int_{\re^3}|x|^2m_\varepsilon(x,t)dx\leq C$.
\end{itemize}
\end{proposition}

%%%%%%%%%%%%%%%%%%%%%%%%

\section{Formulation of the optimal control problem}\label{OC}
In this section, we study equation~\eqref{eq:MFGe}-(i) and the corresponding optimal control problem for any $\varepsilon\in [0,1]$ (note that here we also cope the case $\varepsilon=0$).
We shall show that the value function of this control problem solves~\eqref{eq:MFGe}-(i), is  Lipschitz continuous and semiconcave in~$x$. Moreover, for $\varepsilon\ne0$, we establish an optimal synthesis result exploiting the fact that the matrix $B^\varepsilon(B^\varepsilon)^T$ is invertible.
Throughout this section, since $\varepsilon$ is a fixed parameter, we omit it when it is not essential and we assume the following hypothesis
\begin{hypothesis} \label{BasicAss_illi}
$f\in C^{0}([0,T],C^2(\re^3))$, $g\in C^2(\re^3)$ and there exists a constant $C$ such that
$$\|f(\cdot,t)\|_{C^2(\re^3)} + \|g\|_{C^2(\re^3)} \leq C,\qquad\forall t\in[0,T].$$
\end{hypothesis}
\begin{definition}\label{def:OCDe} We consider the following optimal control problem
\begin{equation}\label{def:OC_illi}
\text{minimize } J_{x,t}(\alpha):
=\displaystyle\int_t^T\dfrac12|\alpha(s)|^2+f( x(s),s)\,ds+g(x(T))
\end{equation}
subject to $(x(\cdot), \alpha(\cdot))\in \mathcal A_{\varepsilon}(x,t)$, where
\begin{equation} \mathcal A_{\varepsilon}(x,t):=\left\{(x(\cdot), \alpha(\cdot))\in AC([t,T]; \R^3)\times L^2([t,T];\re^3):\, \textrm{\eqref{DYNe} holds a.e. with } x(t)= x\right\}.
\end{equation}
A couple $(x(\cdot),\alpha(\cdot))\in\mathcal A_{\varepsilon}(x,t)$ is said to be admissible. We say that $x^*(\cdot)$ is an optimal trajectory if there is a control $\alpha^*(\cdot)$ such that $(x^*(\cdot),\alpha^*(\cdot))\in  \mathcal A_{\varepsilon}(x,t)$ is optimal for the optimal control problem in~\eqref{def:OCDe}; in particular, since $\varepsilon$ is fixed, we do not write explicitly that $(x^*,\a^*)$ depends on $\varepsilon$. Also, we shall refer to the system~\eqref{DYNe} as to the dynamics of the optimal control problem in~\eqref{def:OCDe}.
\end{definition}
\begin{remark} \label{2.3} 
Hypothesis~\ref{BasicAss_illi} ensures that, for any $(x,t)\in \re^3\times(0,T)$, the optimal control problem in definition \ref{def:OCDe} admits a solution $(x^*(\cdot),\alpha^*)$ thanks to the LSC with respect to the weak $L^2$ topology. Moreover, just testing $J_{x,t}(\alpha^*)$ against $J_{x,t}(0)$, we get
\begin{equation}\label{bd:alpa_L2}
\|\alpha^*\|_{L^2(t,T)}\leq C_1:=C[(T-t)+1],
\end{equation}
where $C$ is the constant introduced in Hypothesis~\ref{BasicAss_illi} (note that  $C_1$ is independent of $\varepsilon$).
In particular, by H\"older inequality,
\begin{equation}\label{bd:x_1,x_2}
x^*\in C^{1/2}([t,T], \re^3),
\end{equation}
with an $1/2$-H\"older constant of the type $C(1+|x_1|+|x_2|)$, with $C$ in independent of $\varepsilon$.
\end{remark}

\begin{definition} The value function for the cost $J_{x,t}$ defined in \eqref{def:OCDe} is
\begin{equation}\label{repre}
u_\varepsilon(x,t):=\inf\left\{ J_{x,t}(\alpha):\, (x(\cdot),\alpha(\cdot))\in \mathcal A_{\varepsilon}(x,t)\right\}.
\end{equation}
An optimal couple $(x^*(\cdot), \alpha^*)$ for the control problem in definition \ref{def:OCDe} is also said to be optimal for $u(x,t)$.
\end{definition}

The following proposition ensures that we can restrict our study on locally uniformly bounded controls.

\begin{proposition}\label{boundalfa}
 Let $u_\varepsilon$ be the value function introduced in \eqref{repre}. Then, there exists a constant $C_2$ (depending only on $T$ and on the constant $C$ of Hypothesis \ref{BasicAss_illi}) such that there holds
\begin{equation}\label{restric}
u_\varepsilon(x,t)=\inf \{J_{x,t}(\alpha):\ (x(\cdot), \alpha) \in\mathcal A_\varepsilon(x,t),\ \|\alpha\|_{\infty}\leq C_2(1+|x_1|+|x_2|)\}
\end{equation}
for any $x=(x_1,x_2,x_3)\in \re^3$ and $t\in[0,T]$.
\end{proposition}

\begin{proof}
The idea of the proof is borrowed from \cite[Theorem 2.1]{BCP}. For $x=(x_1,x_2,x_3)\in \re^3$ and $t\in[0,T]$, let $\alpha$ be an optimal control for $u(x,t)$.
For $\mu>0$, let $I_{\mu}:=\{ s\in(t,T): |\alpha(s)|> \mu\}$.
Define 
\begin{equation}\label{alfamu}
\alpha^{\mu}(s)=
\left\{\begin{array}{ll}
\alpha(s) &\text { if }|\alpha(s)|\leq \mu,\\
0 &\text { if }|\alpha(s)|> \mu.
\end{array}\right.
\end{equation}
Let $x^{\mu}(s)$ be the trajectory starting from $x$ associated to the control~$\alpha^{\mu}(s)$.
We claim that 
\begin{equation}\label{claimmu}
|x^{\mu}(s)-x(s)|\leq K(1+|x_1|+|x_2|+\varepsilon)\int_{I_{\mu}} |\alpha(\tau)|d\tau \qquad\forall s\in[t,T]
\end{equation}
where $K$ is a constant depending only on $C_1$ (see \eqref{bd:alpa_L2}) and $T$. Actually, for the first two components of $x^{\mu}(s)-x(s)$ we have 
\begin{equation}\label{primedue}
\left| x_i^{\mu}(s)-x_i(s)\right|\leq \int_t^s \left|\alpha_i^{\mu}(\tau)-\alpha_i(\tau)\right|d\tau= \int_{I_{\mu}} |\alpha_i(\tau)|d\tau \qquad\forall s\in[t,T],\ i=1,2.
\end{equation}
For the third component, there holds
\begin{eqnarray*}
x_3^{\mu}(s)-x_3(s)&=&\int_t^s \left[-x_2^{\mu}(\tau)\alpha_1^{\mu}(\tau)+x_1^{\mu}(\tau)\alpha_2^{\mu}(\tau)+ x_2(\tau)\alpha_1(\tau)-x_1(\tau)\alpha_2(\tau)\right.\\
&&\left.+\varepsilon (\alpha_3^{\mu}(\tau)-\alpha_3(\tau))\right]d\tau\\
&=&\int_t^s \left[(x_2(\tau)-x_2^{\mu}(\tau))\alpha_1^{\mu}(\tau)+ x_2(\tau)(\alpha_1(\tau)-\alpha_1^{\mu}(\tau)) \right.\\ 
&&\left. + (x_1^{\mu}(\tau)-x_1(\tau))\alpha_2^{\mu}(\tau)  + x_1(\tau)(\alpha_2^{\mu}(\tau)-\alpha_2(\tau))+\varepsilon (\alpha_3^{\mu}(\tau)-\alpha_3(\tau))\right] d\tau.
\end{eqnarray*}
Hence from \eqref{bd:x_1,x_2} and \eqref{primedue}, we infer 
\begin{eqnarray*}
|x_3^{\mu}(s)-x_3(s)|&\leq& \int_{I_{\mu}} |\alpha_2(\tau)|d\tau\int_t^s |\alpha_1^{\mu}(\tau)|d\tau + [|x_2|+C_1  (T-t)^{1/2}]
\int_{I_{\mu}} |\alpha_1(\tau)|d\tau\\
&&+\int_{I_{\mu}} |\alpha_1(\tau)|d\tau\int_t^s |\alpha_2^{\mu}(\tau)|d\tau
+ [|x_1|+C_1 (T-t)^{1/2}]
\int_{I_{\mu}} |\alpha_2(\tau)|d\tau\\
&&+\varepsilon\int_{I_{\mu}} |\alpha_3(\tau)|d\tau.
\end{eqnarray*}
Moreover, by H\"older inequality and \eqref{bd:alpa_L2}, we have
\[
\int_t^s |\alpha_i^{\mu}(\tau)|d\tau\leq \sqrt{s-t}\|\alpha\|_2\leq C_1\sqrt{T-t}, \ i=1,2.
\]
Replacing the last inequality in the previous one, we accomplish the proof of the claim~\eqref{claimmu}.

Now, the definition of the cost $J_{x,t}(\alpha)$ in \eqref{def:OC_illi} and the Lipschitz continuity of $f$ and $g$ yield
\begin{eqnarray}
&&J_{x,t}(\alpha^{\mu})- J_{x,t}(\alpha)=\nn\\
&&=\displaystyle\int_t^T\dfrac12|\alpha^{\mu}(s)|^2+f( x^{\mu}(s),s)\,ds+g(x^{\mu}(T))-
\displaystyle\int_t^T\dfrac12|\alpha(s)|^2+f( x(s),s)\,ds-g(x(T))\nn\\
&&\leq-\int_{I_{\mu}}\dfrac12|\alpha(s)|^2ds+ 
L_f\int_t^T|x^{\mu}(s)-x(s)|ds+L_g|x^{\mu}(T)-x(T)|\nn\\
&&\leq \int_{I_{\mu}}\left(-\dfrac12|\alpha(s)|^2+ K(L_f (T-t)+L_g)(1+|x_1|+|x_2|+\varepsilon)|\alpha(s)|\right) ds,\nn
\end{eqnarray}
where the last inequality comes from \eqref{claimmu}.
Hence, if $I_\mu$ has positive measure for $\mu>2 K(L_f (T-t)+L_g)(1+|x_1|+|x_2|+\varepsilon)$, the last integrand is negative for every $s\in I_\mu$ which contradicts the optimality of $\alpha$. This implies that these $I_\mu$ have null measure and, in particular, $\|\alpha\|_\infty\leq 2 K(L_f (T-t)+L_g)(1+|x_1|+|x_2|+\varepsilon)$. Since the choice of $\mu$ is independent of $\varepsilon\in[0,1]$, we get the result.
\end{proof}

\subsection{Necessary conditions and regularity for the optimal trajectories}
The application of the Maximum Principle (see \cite[Theorem 22.17]{Cla})
yields the following necessary conditions.
\begin{proposition}\label{prop:pontriagind}
Let $(x^*, \alpha^*)$ be optimal for the optimal control problem in \eqref{def:OC_illi}. 
Then, there exists an arc $p\in AC([t,T]; \re^3)$, hereafter called the costate,  such that
\begin{enumerate}
\item The pair
$(x^*, p)$ satisfies the system of differential equations for a.e. $s\in [t,T]$
\begin{equation}\label{tage}
\left\{
\begin{array}{ll}
\quad x_1'= p_1-x_2p_3\\
\quad x_2'= p_2+x_1p_3\\
\quad x_3'= (x_1^2+x_2^2+\varepsilon^2)p_3+x_1p_2-x_2p_1\\
\quad p_1'=-(p_2+x_1p_3)p_3+  f_{x_1}(x,s)\\
\quad p_2'=(p_1-x_2p_3)p_3+  f_{x_2}(x,s)\\
\quad p_3'=f_{x_3}(x,s)
\end{array}\right.
\end{equation}
with the mixed boundary conditions
\begin{equation}\label{tag:bcd}
x(t)=x,\quad p(T)=-D g(x(T)).
\end{equation}
\item The optimal control $\alpha^*$ verifies for a.e. $s\in [t,T]$
\begin{equation}\label{tagalphae}
\alpha_1(s)= p_1-x_2p_3,\qquad
\alpha_2(s)= p_2+x_1p_3,\qquad
\alpha_3(s)= \varepsilon p_3.
\end{equation}
\end{enumerate}
\end{proposition}
\begin{remark}\label{rmk:PMPintrin}
Let us observe that equations~\eqref{tage} and \eqref{tagalphae} can be rewritten in terms of the vector fields as follows
\begin{equation*}
\begin{array}{lll}
 x_1'= X_1p,&\quad x_2'= X_2p,&\quad x_3'=-x_2X_1p+x_1X_2p+\varepsilon^2 p_3,\\
 p_1'=-p_3X_2p+ f_{x_1}(x,s),&\quad  p_2'=p_3X_1p +  f_{x_2}(x,s),&\quad p_3'=f_{x_3}(x,s)
\end{array}
\end{equation*}
and respectively
\begin{equation*}
\alpha_1(s)=X_1p(s),\qquad \alpha_2(s)=X_2p(s),\qquad \alpha_3(s)=\varepsilon p_3(s).
\end{equation*}
\end{remark}

\begin{corollary}\label{coro:regularityd}
Let $(x^*, \alpha^*)$ be optimal for the optimal control problem in \eqref{def:OCDe}. Then:
\begin{itemize}
\item [1.] The unique solution of the Cauchy  problem
$$
\left\{
\begin{array}{ll}
\quad \pi_1'=-(\pi_2+x^*_1\pi_3)\pi_3+  f_{x_1}(x^*,s),\\
\quad \pi_2'=(\pi_1-x^*_2\pi_3)\pi_3+  f_{x_2}(x^*,s),\\
\quad \pi_3'=f_{x_3}(x^*,s),\\
\quad \pi(T)=-D g(x^*(T)).
\end{array}\right.
$$
is the costate $p$ associated to $(x^*, \alpha^*)$ as in Proposition~\ref{prop:pontriagind}.
\item[2.] The optimal $\alpha^*$ is a feedback control and it is
uniquely expressed by
\[
\alpha_1^*(s)= p_1-x^*_2p_3,\qquad
\alpha_2^*(s)= p_2+x^*_1p_3,\qquad
\alpha_3^*(s)= \varepsilon p_3
\]
where $p$ is the costate associated to $(x^*, \alpha^*)$.
\item[3.] The functions  $x^*$  and $\alpha^*$ are of class $C^1$. In particular equations \eqref{tage} and \eqref{tagalphae} hold for every $s\in [t,T]$.
\item[4.] Assume that,  for some $k\in\mathbb N$,
$D_xf\in C^k$.
Then, the costate~$p$ and the control~$\alpha^*$ are of class $C^{k+1}$ and $x^*$ is of class $C^{k+2}$.
 \end{itemize}
\end{corollary}
\begin{proof}
Point 1 is an immediate consequence of \eqref{tage} together with the endpoint condition \eqref{tag:bcd}. Point 2 follows from \eqref{tagalphae}.\\
3. Since $x^*$ and $p$ are continuous, the continuity of  $\alpha^*$ follows from \eqref{tagalphae}. The dynamics~\eqref{DYNe} yield that $x^*\in C^1$. 
Relations~\eqref{tage} and~\eqref{tagalphae}  imply that  $p$ and $\alpha^*$  are of class $C^1$.
By a standard bootstrap inductive argument still using \eqref{DYNe}, \eqref{tage} and \eqref{tagalphae}, we get point 4.
\end{proof}

\begin{proposition}\label{OS_e}
For $\varepsilon\neq 0$, the optimal trajectories are unique after the initial time: if $x^*(\cdot)$ is an optimal trajectory for $u_\varepsilon(x,t)$, then for every $t<\tau<T$ there are no other optimal trajectories for $u_\varepsilon(x^*(\tau),\tau)$ other than $x^*(\cdot)$ restricted to $[\tau,T]$. 
\end{proposition}
\begin{proof}
Let~$y^*$ be an optimal trajectory for~$u(x^*(\tau),\tau)$; the concatenation~$z^*$ of~$x^*$ with~$y^*$ at~$\tau$ is still optimal for~$u_\varepsilon(x,t)$. Let~$p$ and~$q$ be respectively the costate of~$x^*$ and of~$z^*$. By point~$(4)$ both~$x^*$ and~$z^*$ are $C^1$.
Since the matrix~$B^{\varepsilon}(x)(B^{\varepsilon}(x))^T$ is invertible, we denote by $\beta(x)$ its inverse and, from the first three lines in~\eqref{tage}, we get
\begin{equation*}
p(s)=\beta(x^*(s))\dot x^*(s)\qquad\textrm{and}\qquad q(s)=\beta(z^*(s))\dot z^*(s)\qquad\forall s\in(t,T).
\end{equation*}
Since $x^*(\cdot)=z^*(\cdot)$ in~$[t,\tau]$, we get: $p(\tau)=q(\tau)$. In conclusion, both~$(x^*,p)$ and~$(z^*,q)$ solve the same Cauchy problem \eqref{tage}
on~$(\tau,T]$ with the same data at time~$\tau$. The Cauchy-Lipschitz theorem ensures that they coincide.
\end{proof}

\subsection{The Hamilton-Jacobi equation and the value function of the optimal control problem}\label{vf}
The aim of this section is to study the Hamilton-Jacobi equation~\eqref{eq:MFGe}-(i) with $m$ fixed, namely
\begin{equation}\label{HJei}
\left\{\begin{array}{ll}
-\partial_t u_\varepsilon+H^\varepsilon(x,Du_\varepsilon)=f(x, t)&\qquad \textrm{in }\re^3\times (0,T),\\
u_\varepsilon(x,T)=g(x)&\qquad \textrm{on }\re^3
\end{array}\right.
\end{equation}
where $H^\varepsilon$ is defined in~\eqref{He}.
Under Hypothesis~\ref{BasicAss_illi}, we shall prove Lipschitz continuity and semiconcavity of $u_\varepsilon$.
As a first step, in the next lemma we show that the solution $u_\varepsilon$ of \eqref{HJei} can be represented as the value function of the control problem defined in \eqref{repre}.

\begin{lemma}\label{3.2}
Under Hypothesis \ref{BasicAss_illi}, the value function~$u_{\varepsilon}$ defined in~\eqref{repre} is the unique bounded viscosity solution to problem~\eqref{HJei} and there exists a constant~$C$ independent of~$\varepsilon$ such that
\begin{equation}\label{boundue}
\|u_{\varepsilon}\|_{\infty}\leq C.
\end{equation}
\end{lemma}
\begin{proof}
The proof comes from classical results: see for instance \cite[Proposition III.3.5]{BCD} and \cite[Theorem 3.1]{BCP}).
The bound of $u_{\varepsilon}$ uniform on $\varepsilon$ is obtained taking as admissible control $\alpha=0$ in \eqref{repre}.
\end{proof}

\begin{lemma}\label{LLS} The value function~$u_\varepsilon$ fulfills the following properties
\begin{enumerate}
\item $u_{\varepsilon}$ is Lipschitz continuous with respect to the spatial variable~$x$ uniformly on $\varepsilon$,
\item $u_{\varepsilon}$ is locally
Lipschitz continuous with respect to the time variable $t$ with a
Lipschitz constant $C(1+|x_1|^2+|x_2|^2)$ where $C$ is a constant independent of~$\varepsilon$.
\end{enumerate}
\end{lemma}
\begin{proof}
In this proof, $C_T$ will denote a constant which may change from line to line but it always depends only on the constants in the assumptions (especially the Lipschitz constants of $f$ and $g$) and on~$T$.
For simplicity we write ``$u$'' instead of ``$u_\varepsilon$''.\\
1. Let $t$ be fixed. We follow the proof of \cite[Lemma 4.7]{C}.
From Remark \ref{2.3} we know that there exists~$\alpha(\cdot)$ optimal control for $u(x,t)$ and $x(\cdot)$ optimal trajectory i.e.:
\begin{equation}
\label{eq:HJ31}
u(x_1, x_2, x_3, t)=\int_t^T\frac12 |\alpha(s)|^2+f(x(s),s)\,ds+g(x(T)).
\end{equation}
We consider the path $x^*(s)$ starting from $y=(y_1,y_2, y_3)$, with control $\alpha$. %Let $\alpha(\cdot)$ be the optimal control for $u(x,t)$ and $x(\cdot)$ the optimal trajectory. Let $x^*(s)$ be the path starting from $y=(y_1,y_2, y_3)$, with control $\alpha(\cdot)$.
%To prove the Lipschitz continuity w.r.t. $x$ uniform on $\varepsilon$ we proceed as in Lemma \ref{L1}; the first two components of the trajectories  $x(s)$ and $x^*(s)$ are as in the previous section and the third components become
We have
\begin{eqnarray*}
x_1^*(s)&=&y_1+\int_t^s\alpha_1(\tau) \,d\tau=y_1-x_1+x_1(s)\\
x_2^*(s)&=&y_2+\int_t^s \alpha_2(\tau)\,d\tau=y_2-x_2+x_2(s)\\
x_3^*(s)&=&y_3-\int_t^s\alpha_1(\tau)x_2^*(\tau)\,d\tau+
\int_t^s\alpha_2(\tau)x_1^*(\tau)\,d\tau +\varepsilon\int_t^s\alpha_3(\tau)\,d\tau\\
&=&y_3- (y_2-x_2)\int_t^s\alpha_1(\tau)\,d\tau+
(y_1-x_1)\int_t^s\alpha_2(\tau)\,d\tau+\varepsilon\int_t^s\alpha_3(\tau)\,d\tau\\
&&+\int_t^s (-\alpha_1(\tau)x_2(\tau)+\alpha_2(\tau)x_1(\tau))\,d\tau\\
&=&x_3(s)+(y_3-x_3)- (y_2-x_2)\int_t^s\alpha_1(\tau)\,d\tau+
(y_1-x_1)\int_t^s\alpha_2(\tau)\,d\tau+\varepsilon\int_t^s\alpha_3(\tau)\,d\tau\\
&=&x_3(s)+(y_3-x_3)- (y_2-x_2)\int_t^s\alpha_1(\tau)\,d\tau+
(y_1-x_1)\int_t^s\alpha_2(\tau)\,d\tau.
\end{eqnarray*}
Using the Lipschitz continuity of $f$ we get
\begin{multline*}
f(x^*(s),s)
\leq f(x(s), s)
+L(|y_1-x_1|+ |y_2-x_2| + |y_3-x_3|+\\ +|y_2-x_2|\sqrt{s-t}\|\alpha_1\|_2 +|y_1-x_1|\sqrt{s-t}\|\alpha_2\|_2)\end{multline*}
and from the $L^2$ uniform estimate for $\alpha_1$ and $\alpha_2$ in \eqref{bd:alpa_L2} we get 
$$f(x^*(s),s)-f(x(s), s)\leq C_T(\vert y_1-x_1\vert+\vert y_2-x_2\vert +\vert y_3-x_3\vert).$$
By the same calculations for $g$ and substituting equality~\eqref{eq:HJ31} in
\begin{equation*}%\label{eq:HJ5}
u(y_1, y_2, y_3, t)\leq \int_t^T\frac12 |\alpha(s)|^2+f(x^*(s),s)\,ds+g(x^*(T)),
\end{equation*}
we get
\begin{equation*}%\label{eq:HJ6}
u(y_1, y_2, y_3, t)\leq u(x_1, x_2, x_3, t)+C_T(\vert y_1-x_1 \vert +\vert y_2-x_2\vert+ \vert y_3-x_3\vert).
\end{equation*} 
Reversing the role of $x$ and $y$ we get the result.

2. For proving the Lipschitz continuity w.r.t. $t$, we note that
\begin{eqnarray*}
|x(s)-x|&\leq& C_T(s-t)(\|\alpha_1\|_{\infty}|x_2|+ \|\alpha_2\|_{\infty}|x_1|+\varepsilon\|\alpha_3\|_{\infty})\\
&\leq& C_T (1+|x_1|^2+|x_2|^2+\varepsilon)(s-t)\\&\leq& C_T (1+|x_1|^2+|x_2|^2),
\end{eqnarray*}
where the second inequality is due to the $L^\infty$-bound for optimal controls established in Proposition \ref{boundalfa}. We accomplish the proof following the same arguments as those in the proof of \cite[Lemma 4.7]{C}.
\end{proof}

In the following lemma we establish the semiconcavity of~$u$ w.r.t. $x$; we recall here below the definition of semiconcavity with linear modulus and we refer the reader to the monograph \cite{CS} for further properties.

\begin{definition}\label{SMC}
Let $u:\re^d\to\re$. We say that $u$ is {\em {semiconcave}} (with linear modulus)
if there exists a constant $C\geq 0$ such that  for all $\lambda\in [0,1]$,
$$\lambda u(y)+(1-\lambda)u(x)-u(\lambda y+(1-\lambda)x)\leq C\lambda(1-\lambda)|y-x|^2\qquad \forall x, y\in \re^d.$$
\end{definition}

\begin{lemma}\label{lemmasemic}
Under Hypothesis~\ref{BasicAss_illi}, the value function~$u_\varepsilon$, defined in \eqref{repre}, is semiconcave w.r.t. $x$ with a linear modulus independent on $\varepsilon$.
\end{lemma}

\begin{proof}
For any~$x,y\in \re^3 $ and $\lambda\in[0,1]$, consider~$x_{\lambda}:=\lambda x+(1-\lambda)y$. Let $\alpha(s)$ and $x_{\lambda}(s)$ be an optimal control and respectively the corresponding optimal trajectory for~$u(x_{\lambda}, t)$; for $s\in[t,T]$ there hold
\begin{eqnarray*}
x_{\lambda,i}(s)&=&x_{\lambda,i}+\int_t^s\,\alpha_i(\tau)\,d\tau,\qquad\textrm{for } i=1,2\\
x_{\lambda,3}(s)&=&x_{\lambda,3}-\int_t^s\alpha_1(\tau)x_{\lambda,2}(\tau)\,d\tau+
\int_t^s\alpha_2(\tau)x_{\lambda,1}(\tau)\,d\tau+\varepsilon\int_t^s\alpha_3(\tau)d\tau.
\end{eqnarray*}
Let $x(s)$ and $y(s)$ satisfy \eqref{DYNH} with initial condition respectively $x$ and $y$ still with the same control $\alpha$, optimal for~$u(x_{\lambda}, t)$.
We have to estimate~$\lambda u(x,t) +(1-\lambda)u(y,t)$ in terms of $u(x_{\lambda}, t)$. To this end, arguing as in the proof of~\cite[Lemma 4.7]{C},  we have to estimate the terms $\lambda f(x(s), s) +(1-\lambda)f(y(s), s)$ and $\lambda g(x(T))+(1-\lambda)g(y(T)).$\\
We explicitly provide the calculations for the third component $x_3(s)$ since the calculations for $x_1(s)$ and $x_2(s)$ are the same as in \cite{C}.
We have
\begin{eqnarray*}
x_3(s)&=&x_3-\int_t^s\alpha_1(\tau)x_2(\tau)\,d\tau+
\int_t^s\alpha_2(\tau)x_1(\tau)\,d\tau+\varepsilon\int_t^s\alpha_3(\tau)d\tau\\
&=&x_3-x_{\lambda,3}+x_{\lambda,3}(s)- \int_t^s\alpha_1(\tau)(x_2(\tau)-x_{\lambda,2}(\tau))\,d\tau+
\int_t^s\alpha_2(\tau)(x_1(\tau)-x_{\lambda,1}(\tau))\,d\tau.
\end{eqnarray*}
Since $x_3-x_{\lambda,3}=(1-\lambda)(x_3-y_3)$ and
\begin{equation}\label{diff1}x_i(\tau)-x_{\lambda,i}(\tau)=(1-\lambda)(x_i-y_i)\qquad \textrm{for }i=1,2,
\end{equation}
we get
\begin{equation}\label{diff}
x_3(s)-x_{\lambda,3}(s)=(1-\lambda)\left[x_3-y_3- (x_2-y_2)\int_t^s\alpha_1(\tau)d\tau+ (x_1-y_1)\int_t^s\alpha_2(\tau)d\tau\right].
\end{equation}
Analogously for $y(s)$: since $y_3-x_{\lambda,3}=\lambda(y_3-x_3)$ and
\begin{equation}\label{diff2}
y_i(\tau)-x_{\lambda,i}(\tau)=\lambda(y_i-x_i)\qquad\textrm{for }i=1,2,
\end{equation}
we get
\begin{equation}\label{diff3}
y_3(s)-x_{\lambda,3}(s)=\lambda\left[(y_3-x_3)+ (x_2-y_2)\int_t^s\alpha_1(\tau)d\tau- (x_1-y_1)\int_t^s\alpha_2(\tau)d\tau\right].
\end{equation}

For the sake of brevity we provide the explicit calculations only for $f$ omitting the analogous ones for $g$; and we write $f(x_1,x_2, x_3):=f(x_1, x_2, x_3, s)$. 
We have
\begin{equation*}
\begin{array}{l}
\lambda f(x(s))+(1-\lambda)f(y(s))=\\
 \lambda f(x_1(s),x_{2}(s), x_{\lambda,3}(s)+(1-\lambda)(x_3-y_3- (x_2-y_2)\int_t^s\alpha_1(\tau)d\tau+ (x_1-y_1)\int_t^s\alpha_2(\tau)d\tau))+\\
+(1-\lambda)f(y_1(s), y_{2}(s), x_{\lambda,3}(s)+\lambda(y_3-x_3+ (x_2-y_2)\int_t^s\alpha_1(\tau)d\tau- (x_1-y_1)\int_t^s\alpha_2(\tau)d\tau).
\end{array}
\end{equation*}
Since for $i=1,2$ there holds
$$\lambda \partial_{x_i}f(x_{\lambda}(s))(x_i(s)-x_{\lambda,i}(s))+(1-\lambda) \partial_{x_i}f(x_{\lambda}(s))(y_i(s)-x_{\lambda,i}(s))=0,$$ 
the Taylor expansion of $f$ centered in  $x_{\lambda}(s)$ gives:
\begin{multline*}
\lambda f(x(s))+(1-\lambda)f(y(s))=\\
\lambda(f(x_{\lambda}(s))+ Df(x_{\lambda}(s))(x(s)-x_{\lambda}(s))
+R_1)+ 
(1-\lambda)(f(x_{\lambda}(s))+ Df(x_{\lambda}(s))(y(s)-x_{\lambda}(s)) +R_2)\\
=\lambda\left(f(x_{\lambda}(s))+ \partial_{x_3}f(x_{\lambda}(s))
(1-\lambda)(x_3-y_3- (x_2-y_2)\int_t^s\alpha_1(\tau)d\tau + (x_1-y_1)\int_t^s\alpha_2(\tau)d\tau) +
R_1\right)\\
+(1-\lambda)\bigg(f(x_{\lambda}(s))+ \partial_{x_3}f(x_{\lambda}(s))
\lambda(y_3-x_3+ (x_2-y_2)\int_t^s\alpha_1(\tau)d\tau- (x_1-y_1)\int_t^s\alpha_2(\tau)d\tau) +R_2\bigg)=\\
=f(x_{\lambda}(s))+ \lambda R_1+(1-\lambda)R_2,
\end{multline*}
where $R_1$ and $R_2$ are the error terms of the expansion, namely
\begin{multline*}
\lambda R_1+(1-\lambda)R_2=
\frac{1}{2}\lambda ( (x(s)-x_{\lambda}(s))D^2f(\xi_1)(x(s)-x_{\lambda}(s))^T\\
+\frac{1}{2}(1-\lambda)( (y(s)-x_{\lambda}(s))D^2f(\xi_2)(y(s)-x_{\lambda}(s))^T,
\end{multline*}
for suitable $\xi_1, \xi_2\in \re^3$.

Using relations~\eqref{diff1}-\eqref{diff3} and the $L^2$ uniform estimate of $\alpha$ in \eqref{bd:alpa_L2} (which is independent of $\varepsilon$), we obtain
\begin{equation*}
\left\{\begin{array}{ll}
|x_i(s)-x_{\lambda,i}(s)|\,|x_j(s)-x_{\lambda,j}(s)|\leq C(1-\lambda)^2|x-y|^2&\qquad i,j=1,2,3\\
|y_i(s)-x_{\lambda,i}(s)|\,|y_j(s)-x_{\lambda,j}(s)|\leq C\lambda^2|x-y|^2&\qquad i,j=1,2,3
\end{array}\right.
\end{equation*}
for some positive constant $C$. Then, using Hypothesis \ref{BasicAss_illi}, possibly increasing $C$, we get
$$
\lambda R_1+(1-\lambda)R_2\leq C\lambda(1-\lambda) |x-y|^2,
$$
and, in particular,
$$\lambda f(x(s))+(1-\lambda)f(y(s))\leq f(x_{\lambda}(s))+ C\lambda(1-\lambda) |x-y|^2$$
which amounts to the semiconcavity of $u$.
\end{proof}

\begin{lemma}[Optimal synthesis]\label{lemma:OSeps}
Let~$u_\varepsilon$ be the unique bounded viscosity solution to~\eqref{HJei} founded in Lemma~\ref{3.2}.
\begin{itemize}
\item[a)] Let $\gamma\in AC([t,T])$ be such that
\begin{equation}\label{hyp_OS}
\textrm{$u_\varepsilon(\cdot, s)$ is differentiable at~$\gamma(s)$ for almost every $s\in(t,T)$}
\end{equation}
and
\begin{equation}\label{1802:rmk2_i}
\dot\gamma(t)=-Du_\varepsilon(\gamma(t),t) B^{\varepsilon}(\gamma(t))(B^{\varepsilon}(\gamma(t)))^T,\qquad \gamma(0)=x. 
\end{equation}
Then, the control law $\alpha(s)=-Du_\varepsilon(\gamma(s),s) B^{\varepsilon}(\gamma(s))$ is optimal for $u_\varepsilon(x,t)$.
\item[b)] For~$\varepsilon\neq 0$, if $u_\varepsilon(\cdot,t)$ is differentiable at $x$, then problem~\eqref{1802:rmk2_i} has a unique solution which moreover coincides with the optimal trajectory. In particular, for a.e. $x$, there exists a unique optimal trajectory for~$u_\varepsilon(x,0)$.
\end{itemize}
\end{lemma}
\begin{proof}
We shall follow the same arguments as those used in~\cite[Lemma 4.11]{C} (see also \cite[Lemma 3.5]{AMMT} and \cite[Lemma 3.5]{MMMT} for similar arguments). So we only illustrate the main novelties in the proof and we refer the reader to those papers for the details.
Since $\varepsilon$ is fixed, for simplicity we write ``$u$'' instead of ``$u_\varepsilon$''.\\
$(a)$. Let $\gamma$ be a curve as in the statement; we claim that $\gamma$ is bounded and Lipschitz continuous. Indeed, the differential equation~\eqref{1802:rmk2_i} reads
\begin{equation*}
\gamma_1'=-u_{x_1}+\gamma_2 u_{x_3},\quad \gamma_2'=-u_{x_2}-\gamma_1 u_{x_3},\quad \gamma_3'=\gamma_2 u_{x_1}-\gamma_1 u_{x_2}-(\gamma_1^2+\gamma_2^2+\varepsilon^2) u_{x_3}.
\end{equation*}
Due to the structure of $B^{\epsilon}$, the first two equations do not depend on $\gamma_3$. We define $\xi:=(\gamma_1,\gamma_2)$. By the Cauchy-Schwarz inequality and the Lipschitz continuity of~$u$ found in Lemma~\ref{LLS}, there exists a constant~$C$ such that $|\xi|'\leq C(|\xi|+1)$.
By Gronwall Lemma, we deduce that $\xi$ is bounded and Lipschitz continuous and, consequently, that $\gamma_1$ and $\gamma_2$ are bounded and Lipschitz continuous. Using these properties and again the Lipschitz continuity of~$u$, in the third component of~\eqref{1802:rmk2_i}, we obtain that also $\gamma_3$ is Lipschitz continuous. Our claim is proved.\\
The Lipschitz continuity of~$\gamma$ and of~$u$ entail that the function $s\mapsto u(\gamma(s),s)$ is Lipschitz continuous. The rest of the proof follows the same arguments of the aforementioned papers and it relies on Lebourg's mean value Theorem, Charath\'eodory Theorem and properties of semiconcave functions.\\
$(b)$. We claim that if $Du_\varepsilon(x,t)$ exists then the set $\{\alpha(t),\ \alpha {\text { optimal for } }u(x,t)\}$ is a singleton and that there holds 
$\alpha(t)=-Du_\varepsilon(x,t) B^{\varepsilon}(x)$.
Indeed, if $\alpha(\cdot)$ is optimal for $u_\varepsilon(x,t)$ and $x(\cdot)$ is the corresponding optimal trajectory, then
$$
u_\varepsilon(x,t)= \int_t^T\frac12 |\alpha(s)|^2+ f(x(s),s)\,ds+g(x(T)),
$$
and
$x(\cdot)$ and $\alpha(\cdot)$ satisfy the necessary conditions for optimality proved in Proposition~\ref{prop:pontriagind}.\\
Take $h\in\re^3$ and consider the solution $y(\cdot)$ of (\ref{DYNe}) starting at point $y=(x_1+h_1, x_2+h_2, x_3-x_2h_1+x_1h_2+\epsilon h_3)$ at time~$t$ and obeying to the control~$\alpha$. 
Using the calculations in the proof of point 1. of Lemma \ref{LLS} we get
\begin{eqnarray*}
&&y_1(s)=h_1+x_1(s),\ y_2(s)=h_2+x_2(s)\\
&&y_3(s)=x_3(s)+x_1h_2-x_2h_1+\epsilon h_3- h_2\int_t^s\alpha_1(\tau)\,d\tau+
h_1\int_t^s\alpha_2(\tau)\,d\tau.
\end{eqnarray*}
and
\begin{equation}\label{diffe}
u_\varepsilon(y,t)-u_\varepsilon(x,t)\leq \int_t^T\,(f(y(s), s)-f(x(s),s))\,ds+g(y(T))-g(x(T)).
\end{equation}
From the the differentiability of $u$ w.r.t. $x$ and the arbitrariness of the components of $h$ we get
$$D u_\varepsilon(x, t)B^{\varepsilon}(x)=(I_1,I_2, I_3)$$
where
\begin{eqnarray*}
I_1&=& \int_t^T f_{x_1}(x(s),s)ds+ \int_t^T f_{x_3}(x(s),s)(\int^s_t\alpha_2(\tau)d\tau-x_2)ds+g_{x_1}(x(T))+\\
&&g_{x_3}(x(T))(\int^T_t\alpha_2(\tau)d\tau-x_2),\\
I_2&=& \int_t^T f_{x_2}(x(s),s)ds+ \int_t^T f_{x_3}(x(s),s)(x_1-\int^s_t\alpha_1(\tau)d\tau)ds+g_{x_2}(x(T))+\\
&&g_{x_3}(x(T))(x_1-\int^T_t\alpha_1(\tau)d\tau),\\
I_3&=& \epsilon\int_t^T f_{x_3}(x(s),s)ds+\epsilon g_{x_3}(x(T)).
\end{eqnarray*}
From \eqref{tage} and \eqref{tag:bcd} we have that $f_{x_3}(x(s),s)= p_3^{\prime}(s)$ and $-g_{x_3}(x(T))= p_3(T)$.
Then, integrating by parts and using~\eqref{tagalphae} and~\eqref{tage}, we get 
$I_1=-\alpha_1(t)$, $I_2=-\alpha_2(t)$ and $I_3=-\alpha_3(t)$, i.e. 
$\alpha(t)=-Du_\varepsilon(x,t) B^{\varepsilon}(x)$.
This uniquely determines the value of $\alpha(\cdot)$ at time~$t$.\\
In particular the value of~$p(t)$ is uniquely determined by relations~\eqref{tagalphae} (here, $\varepsilon\ne0$ is needed); hence, system~\eqref{tage} becomes a system of differential equations with condition at time~$t$ and admits a unique solution by the Cauchy-Lipschitz theorem and $x(\cdot)$ is the unique optimal trajectory starting from $x$ associated to the unique optimal control $\alpha(s)$ given by \eqref{tagalphae}. Then from \eqref{tage} $x(\cdot)$ is a solution of \eqref{1802:rmk2_i} and by point a) is the unique solution. 
\end{proof}

%%%%%%%%%%%%%%%%%%%%%%%%

\section{The continuity equation}\label{sect:c_illi}

This section is devoted to study equation \eqref{eq:MFGe}-(ii), namely, to study
\begin{equation}\label{continuitye_illi}
\left\{
\begin{array}{ll} \partial_t m-
\diver(m\, Du B^{\varepsilon}(x)(B^{\varepsilon}(x))^T)=0
&\qquad \textrm{in }\re^3\times (0,T)\\
m(x,0)=m_0(x) &\qquad \textrm{on }\re^3
\end{array}\right.
\end{equation}
where $u$ is the unique bounded (viscosity) solution to problem
\begin{equation}\label{HJe_illi}
\left\{\begin{array}{ll}
-\partial_t u+\frac12 |DuB^{\varepsilon}(x)|^2=F[\overline{m}(t)](x)&\qquad \textrm{in }\re^3\times (0,T)\\
u(x,T)=G[\overline {m}(T)](x)&\qquad \textrm{on }\re^3
\end{array}\right.
\end{equation}
with~$\overline m\in C^{0}([0,T],\mathcal P_1(\re^3))$ (see Lemma~\ref{3.2} for the existence, uniqueness and boundedness of~$u$) .
Since now on, throughout this section, $\varepsilon\in(0,1]$ and $\overline m$ are fixed.

Now we deal with the existence and uniform estimates of the solution $m$ of \eqref{continuitye_illi}.
%
% Teorema di esistenza per l'equazione di continuita'
%
\begin{theorem}\label{prp:m}
Under assumptions (H1)-(H3), for any $\overline m\in C^{0}([0,T],\mathcal P_{1}(\re^3))$, problem~\eqref{continuitye_illi} has a solution $m$
in the sense of Definition~\ref{defsolmfg}.
The function $m$ belongs to $C^{1/2}([0,T],\mathcal P_1(\re^3))\cap L^\infty(0,T;\mathcal P_2(\re^3))$ with a density in~$L^\infty(\re^3\times(0,T))$.
Moreover, $m$ fulfills~\eqref{ambrosioe}-\eqref{flowe} and there exists a constant~$K$ independent of~$\overline m$ and of~$\varepsilon$ such that
\begin{equation}\label{stimeunif}
\|m\|_\infty\leq K, \qquad {\bf d}_1(m(t_1),m(t_2))\leq K|t_2-t_1|^{1/2}, \qquad \int_{\re^3}|x|^2m(x,t)dx\leq K.
\end{equation}
\end{theorem}
The proof of this theorem is postponed at the end of this section. It borrows some arguments of the proof of~\cite[Proposition 3.1]{MMMT} (see also \cite[Theorem 5.1]{C13} and \cite[Theorem 4.20]{C}). More precisely, we shall use a vanishing viscosity approach with the Euclidean Laplacian in the whole system and a truncation argument only in the continuity equation. The vanishing viscosity approach permits to exploit the well posedness of uniformly parabolic equations while the truncation argument permits to overcome the issue of coefficients which grow ``too much'' as $x\to\infty$.\\
For $\sigma \in(0,1]$, we consider the problem
\begin{equation}\label{eq:MFGeN}
\left\{\begin{array}{lll}
(i)&\quad-\partial_t u-\sigma \Delta u+H^{\varepsilon}(x, Du)=F[\overline m(t)](x)&\quad \textrm{in }\re^3\times (0,T)\\
(ii)&\quad\partial_t m-\sigma\Delta m-\diver  (mDu B^{\varepsilon,N}(x)(B^{\varepsilon,N}(x))^T)=0&\quad \textrm{in }\re^3\times (0,T)\\
(iii)&\quad m(x,0)=m_0(x), u(x,T)=G[\overline m(T)](x)&\quad \textrm{on }\re^3,
\end{array}\right.
\end{equation}
where
\begin{equation*}
B^{\varepsilon,N}(x):=  \begin{pmatrix}
\!\!1& 0& 0\!\\
\!\!0&1& 0\!\\
\!\!-\psi_N(x_2)&\psi_N(x_1)&\varepsilon \!
\end{pmatrix}, \qquad \psi_N(\xi):=\left\{\begin{array}{ll}
\xi &\quad \textrm{if }|\xi|\leq N\\0 &\quad \textrm{if }|\xi|\geq 2 N
\end{array}\right.
\end{equation*}
with $\psi_N\in C^2(\re)$, $|\psi_N(\xi)|\leq |\xi|$ for any $\xi\in\re$,  $\|\psi_N\|_{L^\infty}\leq 2N$,
$\|\psi_N'\|_{L^\infty}+\|\psi_N''\|_{L^\infty}\leq K$ ($K$ independent of $N$).

In order to prove Theorem~\ref{prp:m}, it is expedient to establish several properties of the solution~$(u^\sigma,m^\sigma)$ to system~\eqref{eq:MFGeN}: the following lemmata collect existence, uniqueness and other properties of $u^\sigma$ and respectively $m^\sigma$.

\begin{lemma}\label{1802:lemma1}
There exists a unique bounded classical solution $u^\sigma$ to equation~\eqref{eq:MFGeN}-(i) with terminal condition \eqref{eq:MFGeN}-(iii). Moreover there exists a positive constant~$C$ (independent of $\varepsilon$, $\sigma$, $N$ and $\overline m$) such that
\begin{itemize}
\item[a)] $\|u^\sigma\|_\infty\leq C$
\item[b)] $\|D u^\sigma\|_\infty\leq C$, $|\partial_t u^\sigma(t,x)|\leq C (1+|x_1|^2+|x_2|^2)$
\item[c)] $D^2 u^\sigma\leq C$.
\end{itemize}
\end{lemma}

\begin{proof}
It is enough to invoke classical results on parabolic equations for the existence of a solution and the comparison principle (see for instance, \cite{BU,Lie} and also \cite{DL}) using super- and subsolutions of the form $w^\pm=\pm C(T+1-t)$. In particular, we get point $(a)$.\\
Moreover, since the coefficients of $B^{\varepsilon}$ have linear growth at infinity, the bounded solution $u^\sigma$ of \eqref{eq:MFGeN}-(i) is the value function of a stochastic optimal control problem, namely
\begin{equation}\label{reprsigma}
u_\sigma(x,t)=\min \mathbb{E}\bigg(\int_t^T\left[\frac12 |\alpha(\tau)|^2+F[\overline m_\tau](Y_\tau)\right]\,d\tau+G[\overline m_T](Y_T)\bigg)
\end{equation}
where, in $[t,T]$,  $Y(\cdot)$ obeys to a stochastic differential equation
\begin{equation}
\label{stocNP}
dY=\alpha(t) B^{\varepsilon}(Y_t)^Tdt +\sqrt{2\sigma} dW_{t},\qquad Y_t=x
\end{equation}
 and $W_t$ is a standard $3$-dimensional Brownian motion.
Following the procedure used in Lemma \ref{LLS}, point (1), for the deterministic case, we can prove the locally uniform Lipschitz continuity of $u^\sigma$. Actually when evaluating the difference between two random variables obeying to \eqref{stocNP} with the same control, the contribution of the diffusion term is null. Hence we get the first bound of~$(b)$.
Similarly, still using formula \eqref{reprsigma} and the procedure used in Lemma \ref{LLS}, point (2), 
we get the second bound of~$(b)$; using a procedure as in Lemma \ref{lemmasemic} we get
the uniform semiconcavity of $u^\sigma$, namely point $(c)$.
\end{proof}

\begin{lemma}\label{1802:lemma3}
There exists a unique bounded classical solution $m_{\sigma,N}$ to problem \eqref{eq:MFGeN}-(ii) with initial condition \eqref{eq:MFGeN}-(iii). Moreover, there exists a constant~$C_N$ (independent of~$\varepsilon$, $\sigma$ and~$\overline m$) such that $0<m_{\sigma,N}\leq C_N$ in $\re^3\times (0,T)$.
\end{lemma}
\begin{proof}
Since~$\sigma$ is fixed, for simplicity we write ``$u$'' instead of ``$u_\sigma$''. By the regularity of $u$ (see Lemma~\ref{1802:lemma1}), the equation~\eqref{eq:MFGeN}-(ii) can be written as
\begin{equation}\label{eqsciolta}
\partial_t m-\sigma \Delta m-Dm\cdot(Du B^{\varepsilon,N}(B^{\varepsilon,N})^T)-m \diver(Du B^{\varepsilon,N}(B^{\varepsilon,N})^T)=0.
\end{equation}
By Lemma~\ref{1802:lemma1}-(b), we have 
$\|Du B^{\varepsilon,N}(B^{\varepsilon,N})^T\|_{\infty}\leq CN^2$ and also, using standard estimate on the heat kernel, $\|D^2u\|_\infty\leq C$ (where $C$ depends on $\sigma$). By standard results on parabolic equations, we infer existence and uniqueness of a bounded solution~$m_{\sigma,N}$ to~\eqref{eq:MFGeN}.
From the assumptions on $m_0$ and Harnack inequality  (see for example \cite[Theorem 2.1, p.13]{LSU}) we get that $m_{\sigma,N}(\cdot,t)>0$ for $t>0$.

Let us now prove that $m_{\sigma,N}$ fulfills an upper bound, which is independent of~$\varepsilon$, $\sigma$ and $\overline m$.
To this end, we assume for the moment that there exists a positive constant~$k_N$ (independent of $\varepsilon$, $\sigma$ and $\overline m$) such that
\begin{equation}\label{1802:lemma2}
\diver(Du B^{\varepsilon,N}(B^{\varepsilon,N})^T)\leq k_N.
\end{equation}
By estimate \eqref{1802:lemma2} the positivity of $m_{\sigma,N}$ and assumption \eqref{H4'}, from equation \eqref{eqsciolta} we get
\begin{displaymath}
 \partial_t m-\sigma \Delta m-Dm\cdot(Du B^{\varepsilon,N}(B^{\varepsilon,N})^T)-m k_N\leq 0,\qquad m(x, 0)\leq C.  
\end{displaymath}
Then, using the comparison principle with a supersolution of the form $w=C'e^{ C' t}$ (with $C'$ depending only on $k_N$), we obtain that $m_{\sigma,N}\leq C_N$.\\
It remains to prove~\eqref{1802:lemma2}. We denote by $I$ the left hand side. There holds
\begin{eqnarray*}
I&=&\partial_{11} u-2\psi_N(x_2)\partial_{13}u+\partial_{22}u+2\psi_N(x_1)\partial_{23}u+(\psi_N(x_2)^2+\psi_N(x_1)^2+\varepsilon^2)\partial_{33}u\\
&=& \xi_1 D^2u\xi_1^T+ \xi_2 D^2u\xi_2^T+ \xi_3 D^2u\xi_3^T
\end{eqnarray*}
where $\xi_1=(1,0,-\psi_N(x_2))$, $\xi_2=(0,1,\psi_N(x_1))$ and $\xi_3=(0,0,\varepsilon)$. By Lemma~\ref{1802:lemma1}-(c), we obtain
\[
I\leq C(2+2\|\psi_N\|_\infty^2 +\varepsilon^2)\leq C(3+8N^2)
\]
where the last inequality is due to our assumption on $\psi_N$. Choosing $k_N=C(3+8N^2)$ we accomplish the proof of~\eqref{1802:lemma2}.
\end{proof}

\begin{lemma}\label{1802:lemma4}
There exists a constant $K_N$, depending on~$N$ (but independent of~$\varepsilon$, $\sigma$ and $\overline m$) such that the function~$m_{\sigma,N}$ found in Lemma~\ref{1802:lemma3} verifies
\begin{itemize}
\item[a)] ${\bf d}_1(m_{\sigma,N}(t_1),m_{\sigma,N}(t_2))\leq K_N(t_2-t_1)^{1/2}$ for any $0\leq t_1<t_2\leq T$
\item[b)] $\int_{\re^3}|x|^2m_{\sigma,N}(x,t)dx\leq K_N(\int_{\re^3}|x|^2dm_0(x)+1)$.
\end{itemize}
\end{lemma}
\begin{proof}
We take account that $m_{\sigma,N}$ can be interpreted as the law of the following stochastic process
\begin{equation}\label{stocNPm}
d Y^x_t= -Du (Y^x_t,t) B^{\varepsilon,N}(Y^x_t)(B^{\varepsilon,N}(Y^x_t))^T\, dt +\sqrt{2\sigma} d W_t,\qquad Y^x_0=Z_0,
\end{equation}
where ${\mathcal L}(Z_0)=m_0$.\\
By standard arguments, setting
\begin{equation}
\label{mstoch}
m_{\sigma,N}(t):={\mathcal L}(Y_t),
\end{equation}
we know that $m_{\sigma,N}(t)$ is absolutely continuous with respect to Lebesgue measure, and that if $m_{\sigma,N}(\cdot,t)$ is the density of $m_{\sigma,N}(t) $, then $m_{\sigma,N}$   is the  weak solution to~\eqref{eq:MFGeN}-(ii) with $m_{\sigma,N}|_{t=0}=m_0$ (from Ito's Theorem, since the drift is bounded, Proposition 3.6 Chapter 5 \cite{KS}, p.303, and the book \cite{Kr}). Here we have used the bound on $\|Du^\sigma\|$ given in Lemma \ref{1802:lemma1}.\\
b) Noting that
    $$\int_{\re^{3}}|x|^2 dm_{\sigma,N}(t)(x)= 
     {{\mathbb E}}(|Y_t|^2),$$ 
     the desired estimate
    can be obtained applying \cite[Estimate 3.17, p.306]{KS}(see also p. 389) .\\
a)
For $t_2\geq t_1$, it is well known (for instance, see \cite[Lemma 3.4 (proof)]{C}) that
$${\bf d}_1(m_{\sigma,N}(t_1),m_{\sigma,N}(t_2))\leq  {{\mathbb E}}(|Y_{t_1}-Y_{t_2}|).$$
Using this inequality and the boundedness of the drift term in \eqref{stocNPm} with a constant $C_N$ we get
\begin{displaymath}
  \begin{array}[c]{rcl}
&&\mathbb{E}(|Y_{t_1}-Y_{t_2}|)\leq  \ds  \mathbb{E}\left(\int_{t_1}^{t_2} C_N d\tau+ \sqrt {2\sigma} |B_{t_2}-B_{t_1}|\right)
\leq K_N\sqrt{t_2-t_1},
\end{array}
\end{displaymath}
where we have used estimate \cite[(3.17) p. 306]{KS} to the term $|B_{t_2}-B_{t_1}|$.
\end{proof}
\vskip .2cm
Now we let $\sigma\rightarrow 0$ and we consider the problem
\begin{equation}\label{continuityeN}
\left\{
\begin{array}{ll} \partial_t m-
\diver(m\, Du B^{\varepsilon,N}(x)(B^{\varepsilon,N}(x))^T)=0
&\qquad \textrm{in }\re^3\times (0,T)\\
m(x,0)=m_0(x) &\qquad \textrm{on }\re^3
\end{array}\right.
\end{equation}
where $u$ is the unique bounded solution to problem~\eqref{HJe_illi}.
\begin{lemma}\label{1802:thm1}
For $N$ sufficiently large, problem~\eqref{continuityeN} admits exactly one solution~$m_{N}$ in the space $C^{1/2}([0,T],\mathcal P_1(\re^3))\cap L^\infty(0,T;\mathcal P_2(\re^3))$.\\
Moreover, the solution~$m_{N}$ has a density in $L^\infty(\re^3\times(0,T))$ and it is the image of the initial distribution through the flow
\begin{equation}\label{1802:rmk2}
\dot\gamma(t)=-Du(\gamma(t),t) B^{\varepsilon}(\gamma(t))(B^{\varepsilon}(\gamma(t)))^T,\qquad \gamma(0)=x
\end{equation}
which is uniquely determined for $m_0$-a.e. $x\in\re^3$.
\end{lemma}
\begin{proof}
Fix $N$ sufficiently large (it will be suitably chosen later on). For any $\sigma\in(0,1]$, let $(u_\sigma,m_{\sigma,N})$ be the unique classical bounded solution to system~\eqref{eq:MFGeN} (see Lemma~\ref{1802:lemma1} and Lemma~\ref{1802:lemma3}).
Letting $\sigma\to0^+$, by the estimates in Lemma~\ref{1802:lemma1}, in Lemma~\ref{1802:lemma3} and in Lemma~\ref{1802:lemma4}, possibly passing to a subsequence that we still denote $(u_\sigma,m_{\sigma,N})$, we get that the functions~$u_\sigma$ converge locally uniformly to the unique solution~$u$ to~\eqref{HJe_illi} while~$m_{\sigma,N}$ converge to a function $m_{N} \in C^{1/2}([0,T],\mathcal P_1(\re^3))\cap L^\infty(0,T;\mathcal P_2(\re^3))$ in the $C^{0}([0,T],\mathcal P_1(\re^3))$-topology and in the weak-$*$ topology of $L^\infty(\re^3\times(0,T))$. By the same arguments as those in
~\cite[Theorem 4.20 (proof)]{C},
in particular the uniform semiconcavity, we obtain that $m_{N}$ is a solution to~\eqref{continuityeN}.\\
Let us now establish uniqueness and representation formula for the solution~$m_{N}$. We observe that, by Lemma~\ref{1802:lemma1}-(b), the drift verifies $\|Du B^{\varepsilon,N}(B^{\varepsilon,N})^T\|_{\infty}\leq CN^2$ and in particular condition~\cite[eq. (8.1.20)]{AGS} is fulfilled. Hence, we can apply the superposition principle in~\cite[Theorem 8.2.1]{AGS}: there exists a measure $\eta_N$ on $\re^3\times \Gamma$ such that
\begin{itemize}
\item[1)] $m_N(t)=e_t\#\eta_N$ for all $t\in(0,T)$ (recall: $e_t(x,\gamma)=\gamma(t)$)
\item[2)] $\eta_N =\int_{\re^3}(\eta_N)_x dm_0(x)$ where, for $m_0$-a.e. $x\in\re^3$, the measure $(\eta_N)_x$ is concentrated on the set of pairs $(x,\gamma)\in\re^3\times \Gamma$ where $\gamma$ solves
\begin{equation}\label{1802:rmk2N}
\dot\gamma(t)=-Du(\gamma(t),t) B^{\varepsilon,N}(\gamma(t))(B^{\varepsilon,N}(\gamma(t)))^T\quad\textrm{a.e. }t\in(0,T),\qquad \gamma(0)=x.
\end{equation}
\end{itemize}
We now claim that, for $m_0$-a.e. $x\in\re^3$, the solutions to~\eqref{1802:rmk2N} coincide with those of~\eqref{1802:rmk2}. 

The first two components of system \eqref{1802:rmk2N} satisfy
\begin{equation*}
\gamma_1'=-u_{x_1}+\psi_N(\gamma_2) u_{x_3},\quad \gamma_2'=-u_{x_2}-\psi_N(\gamma_1) u_{x_3}.
\end{equation*}
Since the cut-off function $\psi_N$ grows at most linearly, using arguments similar to those in the proof of Lemma~\ref{lemma:OSeps}-(a), we obtain that any solution~$\gamma$ to~\eqref{1802:rmk2N} is bounded uniformly in~$N$.
Since $m_0$ has compact support, there exists a positive constant~$k$ (independent of~$N$) such that, for $m_0$-a.e. $x\in\re^3$, any solution~$\gamma$ to~\eqref{1802:rmk2N} verifies $\gamma(t)\in [-k,k]^3$ for any $t\in(0,T)$. Hence, choosing $N\geq k$, we get that problem~\eqref{1802:rmk2N} coincides with~\eqref{1802:rmk2} if $x\in\textrm{supp}(m_0)$.\\
It remains only to prove that, for $m_0$-a.e. $x\in\re^3$, problem~\eqref{1802:rmk2} admits exactly one solution. To this end, it is enough to invoke~Lemma~\ref{lemma:OSeps} and taking into account that~$u(\cdot,0)$ is Lipschitz continuous.
\end{proof}
Now we let $N\rightarrow +\infty$ and we establish Theorem \ref{prp:m} exploiting that $m_N$ have compact support independent on $N$ which in turn is due to compactness of~$\textrm{supp} (m_0)$.

\begin{proof}[Theorem \ref{prp:m}]
By Lemma~\ref{1802:thm1} all the solutions~$m_{N}$ to~\eqref{continuityeN} coincide if $N$ is sufficiently large. Hence, passing to the limit as $N\to\infty$, we obtain that problem~\eqref{continuitye_illi} admits a solution $m$ in the space $C^{1/2}([0,T],\mathcal P_1(\re^3))\cap L^\infty(0,T;\mathcal P_2(\re^3))$ with a density in~$L^\infty(\re^3\times(0,T))$. Finally, the estimates follow from the corresponding estimates in Lemma~\ref{1802:lemma4}.
\end{proof}

\section{Proof of Proposition \ref{mainproe} and Theorem~\ref{thm:main_illi}}\label{sect:pr_illi}
\begin{proof}[Proposition \ref{mainproe}]
We achieve the proof through a fixed point argument as in~\cite[Theorem 1.1 (proof)]{MMMT} or in \cite[Theorem 2.1]{AMMT}
taking advantage of the results of Lemma~\ref{1802:thm1}. We sketch here some detail of the proof omitting the index $\epsilon$ since it is fixed.
Consider the set
$$
{\cal C} :=\{m\in C^{1/2}([0,T], {{\cal P}_1(\re^3)});\, m(0)=m_0\}
$$ endowed with the norm of~$C^0([0,T]; {\mathcal P}_1(\re^3))$.
Observe that it is a nonempty closed and convex subset of~$C^0([0,T]; {\mathcal P}_1(\re^3))$.
We introduce a map ${\cal T}$ as follows: to any $m\in {\cal C}$ we associate the solution~$u$ to problem~\eqref{HJei} with $f(x,t)=F[m(t)](x)$
and $g(x)=G[m(T)](x)$ and to this $u$ we associate the solution~$\mu=:{\cal T}(m)$ to problem \eqref{continuitye_illi}.\\
By Theorem \ref{prp:m}, the function ${\cal T}(m)$ belongs to ${\cal C}$ hence ${\cal T}$ maps~${\mathcal C}$ into itself.
We claim that the map~${\cal T}$ has the following properties:
\begin{itemize}
\item[(a)] ${\cal T}$ is a continuous map with respect to the norm of~$C^0([0,T]; {\mathcal P}_1)$
\item[(b)] ${\cal T}$ is a compact map.
\end{itemize}
(a) It suffices to follow the same arguments as those in \cite[Lemma 4.19]{C} or in \cite[Theorem 2.1]{AMMT}.\\
(b) Since ${\mathcal C}$ is closed, it is enough to prove that ${\cal T}({\mathcal C})$ is a precompact subset of $C^0([0,T]; {\mathcal P}_1)$. Let~$(\mu_n)_n$ be a sequence in~${\cal T}({\mathcal C})$ with $\mu_n={\cal T}(m_n)$ for some~$m_n\in{\mathcal C}$;
we wish to prove that, possibly for a subsequence,  $\mu_n$ converges to some $\mu$ in the $C^0([0,T]; {\mathcal P}_1(\R^{3}))$-topology as $n\to\infty$.
The functions~${\cal T}(m_n)$ satisfy the estimates \ref{stimeunif} of Theorem \ref{prp:m} with a constant independent of $n$. 
Since the subsets of  $\mathcal P_1$ whose elements have uniformly bounded second moment are relatively compact in $\mathcal P_1$ (see \cite[Lemma 5.7]{C}), Theorem \ref{prp:m}
ensures that the sequence $({\cal T}(m_n))_n$  is uniformly bounded in
$C^{1/2}([0,T];\mathcal P_1)$ and  $L^\infty(0,T;{\mathcal P}_2)$.
Hence we obtain that, possibly for a subsequence (still denoted by ${\cal T}(m_n)$), ${\cal T}(m_n)$ converges to some~$\mu$ in the $C^0([0,T];{\mathcal P}_1(\R^{3}))$-topology.\\
Invoking Schauder fixed point Theorem, we accomplish the proof of the existence of a solution $(u, m)$ of system \eqref{eq:MFGe}.\\
To get \eqref{ambrosioe} we note that,
by Lemma~\ref{1802:thm1} and Theorem \ref{prp:m}, all the solution $m_N$ of the truncated problem \ref{continuityeN} coincide with $m$ if $N$ is sufficiently large. Hence still from Lemma~\ref{1802:thm1}, if $(u,m)$ is a solution of \eqref{eq:MFGe}, then
 for any function $\phi\in C^0_0(\re^3)$, we have
\begin{equation}\label{reprfor}
\int_{\re^3} \phi\, dm(t)=\int_{\re^3}\phi(\overline{ \gamma}_x(t))m_0(x)\, dx
\end{equation}
where $\overline{\gamma}_x$ is the solution of \eqref{flowe} and it is uniquely defined for any $x\in\R^3$.\\
Uniform estimates a) and b) come from Lemma \ref{1802:lemma1} and Lemma \ref{1802:lemma4} noting that $m_N=m$ for $N$ sufficiently large.
\end{proof}

%%%%%%%%%%%%%%

Now, we can prove Theorem~\ref{thm:main_illi}.
\begin{proof}[Theorem \ref{thm:main_illi}]
1. The uniform estimates for~$u_{\varepsilon}$ and for~$m_{\varepsilon}$ in Proposition~\ref{mainproe} ensure that there exist two subsequences, which we will still denote $u_{\varepsilon}$ and respectively $m_{\varepsilon}$ such that, as $\epsilon\to 0^+$, $u_{\varepsilon}$ converge to some function~$u$ locally uniformly in $(x,t)$ and $m_\varepsilon$ converge to some $m\in C^0([0,T], \mathcal P_1(\re^3))$ in the $C^0([0,T],\mathcal P_1(\re^3))$-topology and in the weak-$*$-$L^\infty_{loc}(\re^3\times(0,T))$ topology.
In particular, we get $m(0)=m_0$ and we deduce that $u$ is Lipschitz continuous in~$x$, locally Lipschitz continuous in~$t$, semiconcave in $x$ and
$Du_{\varepsilon}\to Du$ a.e. (because of the semiconcavity estimate in Proposition~\ref{mainproe}-(a) and \cite[Theorem 3.3.3]{CS}).

Being a solution to \eqref{eq:MFGe}-(ii), the function $m_{\varepsilon}$ fulfills
\begin{equation}
\int_0^T\int_{\mathbb{H}^{1}}m_{\varepsilon}(-\partial_t\varphi+D_{\cH}u_{\varepsilon}\cdot D_{\cH}\varphi)dxdt=0\qquad \forall \varphi\in C_{c}^{\infty}(\re^3\times (0,T)).
\end{equation}
Passing to the limit as $\varepsilon\to0^+$ we get that $m$ is a solution to \eqref{eq:MFGintrin}-(ii). On the other hand, by standard stability results for viscosity solutions, we obtain that $u$ is a viscosity solution to~ \eqref{eq:MFGe}-(i).\\

2. Consider the function $m$ found in point (i). Since $t\rightarrow m_t$ is narrowly continuous, applying \cite[Theorem 8.2.1]{AGS}, we get that there exists a probability measure $\eta^*$ in $\re^3\times\Gamma$ which satisfies points (i) and (ii) of \cite[Theorem 8.2.1]{AGS}.
We denote $\eta\in {\mathcal P}(\Gamma)$ the measure on $\Gamma$ defined as $\eta(A):=\eta^*(\re^3\times A)$ for every $A\subset \Gamma$ measurable. We claim that $\eta$ is a MFG equilibrium. Indeed, by \cite[equation (8.2.1)]{AGS}, we have $e_0\#\eta=m_0$ so $\eta\in {\mathcal P}_{m_0}(\Gamma)$. On the other hand, by \cite[Theorem 8.2.1]{AGS}-(i), $\eta$ is supported on the curves solving \eqref{1802:rmk2_i}. From Lemma \ref{lemma:OSeps} such curves are optimal, i.e. belong to the set $\Gamma^{\eta}[x]$, hence our claim is proved.\\
Let us now prove that $(u,m)$ is a mild solution. By \cite[Theorem 8.2.1]{AGS}, we have $m_t=e_t\#\eta$. Moreover, by Lemma \ref{3.2}, the function $u$ found in point~(i) is the value function associated to $m$  as in Definition \ref{mild}-(ii). In conclusion $(u,m)$ is a mild solution to \eqref{eq:MFGintrin}.
\end{proof}

%%%%%%%%%%%%%%%%%%%%%

\section{Generalizations}\label{ext}

\subsection{Some structures of Heisenberg type}

In this section we generalize the previous results to some structures of Heisenberg type (see \cite[Definition 18.1.1]{BLU} for the precise definition and \cite[Theorem 18.2.1]{BLU} for a useful characterization). Throughout this section, the state space is $\re^n$, with $n\geq 1$, and, for any $x=(x_1,\dots,x_n)$, the matrix $B=B(x)\in M^{n\times m}$ (for some $m\leq n$) has the form
\[
\begin{pmatrix}
\!\!&h_{11}& 0&0&\dots&0\!\\
\!\!&h_{21}(x_1)&h_{22}(x_1)&0&\dots&0\!\\
\!\!&h_{31}(x_1,x_2)&h_{32}(x_1,x_2)&h_{33}(x_1,x_2)&\dots&0\!\\
\!\!&\vdots&\vdots&\vdots&\ddots&\vdots\!\\
\!\!&h_{m1}(x_1,\dots, x_{m-1})&h_{m2}(x_1,\dots, x_{m-1})&h_{m3}(x_1,\dots, x_{m-1})&\dots&h_{mm}(x_1,\dots, x_{m-1})\!\\
\!\!&h_{(m+1)1}(x_1,\dots, x_{m})&h_{(m+1)2}(x_1,\dots, x_{m})&h_{(m+1)3}(x_1,\dots, x_{m})&\dots&h_{(m+1)m}(x_1,\dots, x_{m})\!\\
\!\!&\vdots&\vdots&\vdots&\ddots&\vdots\!\\
\!\!&h_{n1}(x_1,\dots, x_{m})&h_{n2}(x_1,\dots, x_{m})&h_{n3}(x_1,\dots, x_{m})&\dots&h_{nm}(x_1,\dots, x_{m})\!
\end{pmatrix}.
\]
In other words, the coefficients of the matrix $B$ fulfills:
\begin{equation}\label{struct_gen}
\left\{\begin{array}{l}
h_{11}\in\re\setminus\{0\}; \qquad \textrm{for } i\in\{1,\dots,m\} \textrm{ and }j>i,\quad h_{ij}=0;\\
\textrm{for }i\in\{2,\dots,m\},\quad h_{ij}(x)=h_{ij}(x_1,\dots,x_{i-1});\\
\textrm{for }i\in\{m+1,\dots,n\},\quad h_{ij}(x)=h_{ij}(x_1,\dots,x_{m}).
\end{array}\right.
\end{equation}

We require that the coefficients $h_{ij}$ fulfill the following hypotheses
\begin{itemize}
\item[(H4)] $h_{ij}\in C^2(\re^n)$ are globally Lipschitz continuous (in particular they have an at most linear growth at infinity) with $D^2h_{ij}$ bounded; the $h_{ij}$'s fulfill~\eqref{struct_gen};
\item[(H5)] there exists a constant $C>0$ such that $\|h_{ij}\|_\infty\leq C$ for any $i=1,\dots,\min\{n-1,m\}$, $j\in\{1,\dots,i\}$;
\item[(H6)] $\{x\in\re^n:\, h_{11}h_{22}\dots h_{mm}=0\}$ has null measure in $\re^n$.
\end{itemize}
Let us observe that the boundedness assumption in $(H5)$ only concerns the first $\min\{n-1,m\}$ rows of $B$; hence, when $n=m$, the coefficients of the $m$-th row can be unbounded, while, when $n>m$, the coefficients of any $i$-th row with $i>m$ can be unbounded.
Let us also emphasize that we do not require that the column of $B$ fulfill the H\"ormander condition and neither assumption 2) of \cite[Theorem 18.2.1]{BLU}.

Now, the generic player in the MFG aims at choosing a control $\alpha=(\alpha_1,\dots,\alpha_m)\in L^2(0,T;\re^m)$ so to minimize the cost in~\eqref{Jgen} when its dynamics $x(\cdot)=(x_1(\cdot), x_2(\cdot),\dots,x_n(\cdot))$ obeys to the differential equation
\begin{equation}\label{dyn_gen}
x_k'(s)=\sum_{j=1}^{\min\{k,m\}}h_{kj}(x(s))\alpha_j(s)\qquad \textrm{with } k=1,\dots,n.
\end{equation}

\begin{corollary}\label{corollary}
Assume hypotheses $(H1)$-$(H6)$. Then, the results of Theorem~\ref{thm:main_illi} hold true.
\end{corollary}
\begin{proof}
The proof is just an adaptation with some heavy calculations of the proof of Theorem~\ref{thm:main_illi} so we only describe the main changes. 

As before, for $\varepsilon\in (0,1]$, we introduce the approximating problem~\eqref{eq:MFGe} with the matrix
\[
B^\varepsilon:=\begin{pmatrix}
\!\!&h_{11}&\dots&0&0&\dots&0\!\\
\!\!&h_{21}(x_1)&\dots&0&0&\dots&0\!\\
\!\!&\vdots&\ddots&\vdots&\vdots&\dots&0\!\\
\!\!&h_{m1}(x_1,\dots, x_{m-1})&\dots&h_{mm}(x_1,\dots, x_{m-1})&0&\dots&0\!\\
\!\!&h_{(m+1)1}(x_1,\dots, x_{m})&\dots&h_{(m+1)m}(x_1,\dots, x_{m})&\varepsilon\!&\dots&0\!\\
\!\!&\vdots&\ddots&\vdots&0&\ddots&0\!\\
\!\!&h_{n1}(x_1,\dots, x_{m})&\dots&h_{nm}(x_1,\dots, x_{m})&0&\dots&\varepsilon\!
\end{pmatrix}\in M^{n\times n}
\]
where the first $m$ columns contain the matrix $B$ while the last $(n-m)$ columns have coefficients
\[
h_{ij}=\varepsilon \qquad\textrm{if }i=j\in\{m+1,\dots,n\},\qquad h_{ij}=0\qquad \textrm{otherwise}.
\]
Explicitely, for $p=(p_1,\dots,p_n)$, the Hamiltonian $H^\varepsilon$ and the drift $\partial_p H^\varepsilon$ are respectively
\begin{eqnarray*}
H^\varepsilon(x,p)&=&\frac12\sum_{i=1}^m\left(\sum_{k=i}^n h_{ki}p_k\right)^2 + \frac{\varepsilon^2}2\sum_{i=m+1}^n p_i^2\\
\frac{\partial H^\varepsilon}{\partial p_j} &=&
\left\{\begin{array}{ll}
\sum_{i=1}^jh_{ji}\left(\sum_{k=i}^n h_{ki}p_k\right)&\quad \textrm{if }j\in\{1,\dots,m\}\\
\sum_{i=1}^m h_{ji}\left(\sum_{k=i}^n h_{ki}p_k\right)+\varepsilon^2 p_j &\quad \textrm{if }j\in\{m+1,\dots,n\}.
\end{array}\right.
\end{eqnarray*}
Assumption (H6) guarantees the uniqueness of the optimal trajectory after the initial time, following the same argument as in Proposition \ref{OS_e}.

The rest of the proof relies on the structure of $B$, namely on these facts: the first $m$ rows have a diagonal form, the coefficients of the $i$-th row, with $i\leq m$ only depend on $x_1,\dots, x_{i-1}$, the coefficients of the last $(n-m)$ rows only depend on $x_1,\dots,x_m$. Indeed, these properties permit to proceed iteratively on the first $m$ coordinates and, afterwards, to study the last $(n-m)$ ones.

For the sake of completeness let us detail the adaptation of the optimal synthesis of Lemma~\ref{lemma:OSeps}. To prove point a) it suffices to establish that the solution to the differential equation~\eqref{1802:rmk2_i} is bounded and Lipschitz continuous (note that also in this case $Du$ is bounded). Indeed, the first $m$ coordinates in~\eqref{1802:rmk2_i} read
\begin{equation}\label{OS_gen1}
\gamma_j'=\sum_{i=1}^j h_{ji}\left(\sum_{k=i}^n h_{ki} u_{x_k} \right).
\end{equation}
We split our arguments according to the fact that $n=m$ or $n>m$.\\
For $n=m$, the structure of $B$ ensures that all the coefficents $h_{ij}$ are independent of $x_m$. In particular, we deduce that the system of differential equations~\eqref{OS_gen1} with $j=1,\dots,m-1$ is independent of $\gamma_m$ and, by $(H5)$, gives an estimate of the form $|\xi'|\leq C(|\xi|+1)$ for $\xi:=(\gamma_1,\dots,\gamma_{m-1})$. By standard theory, we deduce that, for $i\in\{1,\dots,m-1\}$, the curves $\gamma_i$, are bounded and Lipschitz continuous. Finally, we plug this bound in \eqref{OS_gen1} with $j=m$ and we obtain that also $\gamma_m$ is bounded and Lipschitz continuous.\\
For $n>m$, assumption $(H5)$ ensures that $h_{ij}$ are all bounded for $j\in\{1,\dots,m\}$. We deduce that in the system of \eqref{OS_gen1} with $j=1,\dots,m$ all the right hand sides are Lipschitz continuous in $(\gamma_1,\dots,\gamma_{m})$. In particular we obtain that, for $j=1,\dots,m$, all the $\gamma_j$ are bounded and Lipschitz continuous. Afterwards, we consider the system of \eqref{OS_gen1} with $j=m+1,\dots,n$ and, again, we have that all the right hand sides are bounded. Hence, $\gamma$ is bounded and Lipschitz continuous.\\
To get point b) of Lemma~\ref{lemma:OSeps} we choose as increment $\xi=(\xi_1,\dots,\xi_n)\in\re^n$ such that $\xi_k=\sum_{j=1}^nh^\varepsilon_{kj}(x)\beta_j$ where the $\beta_j$'s are arbitrary and the $h^\varepsilon_{kj}(x)$'s are the coefficients of $B^\varepsilon (x)$, namely  $B^\varepsilon=(h^\varepsilon_{kj})_{kj}$.
Now, the necessary conditions read:
\begin{equation}\label{MP_gen}
\left\{
\begin{array}{ll}
(i)\quad x'= p\,B^{\varepsilon}(x)B^{\varepsilon}(x)^T,\\
(ii)\quad p'=- \dfrac{D\left(|p\,B^{\varepsilon}(x)|^2\right)}2+  D f(x,s)\\
(iii)\quad x(t)=x,\quad p(T)=-D g(x(T))
\end{array}\right.
\end{equation}
and
\begin{equation}\label{MP_gen2}
\alpha(s)= p(s)\,B^{\varepsilon}(x(s)),\ \text{ a.e on }[t,T].
\end{equation}
We observe that, by equations \eqref{dyn_gen} and \eqref{MP_gen2}, there holds $\partial_{x_i} \left(|p\,B^{\varepsilon}(x(s))|^2\right)= 2p\partial_{x_i} (x'(s))$ and consequently, equation~\eqref{MP_gen}-(ii) can be written as
\[
\partial_{x_i} f(x(s),s) = p_i'(s)+p\partial_{x_i} (x'(s))\qquad \textrm{for }i=1,\dots,n.
\]
One can follow the same arguments as in the proof of point b) of Lemma~\ref{lemma:OSeps}: integrating by parts and using the last relation and the final condition in \eqref{MP_gen}-(iii), one gets that $Du_{\varepsilon}(x,t)B^{\varepsilon}(x)=-\alpha(t)$ which, together with assumption $(H6)$, yields the uniqueness of the optimal trajectory.
\end{proof}
Let us now give some applications of the above result.
\begin{example}{\it $d$-dimensional Heisenberg.} In the state space $\re^{2d+1}$, with $d\in\mathbb N\setminus\{0\}$, consider the matrix
\[
B(x)=\begin{pmatrix}
\!\!&1&0&\dots&\dots&\dots&0\!\\
\!\!&0&1&0&\dots&\dots&0\!\\
\!\!&0&0&1&0&\dots&0\!\\
\!\!&0&0&0&1&\dots&0\!\\
\!\!&\vdots&\vdots&\vdots&\vdots&\ddots&\vdots\!\\
\!\!&0&0&0&0&\dots&1\!\\
\!\!&-x_{d+1}&\dots&-x_{2d}&x_1&\dots&x_d\!%\\
\end{pmatrix}
\]
where the first $2d$ rows contain the identity matrix.
This matrix fulfills assumptions $(H4)$-$(H6)$ and encompasses the matrix in~\eqref{Hd} when $d=1$.
\end{example}

\begin{example}{\it Completely degenerate case.} In the state space $\re^{n}$, consider the matrix
\[
B(x)=\begin{pmatrix}
\!\!&I_{m}\!\\
\!\!&0_{(n-m),m}\!\\
\end{pmatrix}
\]
where $I_{m}$ is the identity matrix $m\times m$ while $0_{(n-m),m}$ is the null matrix $(n-m)\times m$. With this matrix, the generic player in the MFG controls only its first $m$ coordinates.
This matrix fulfills assumptions $(H4)$-$(H6)$.
\end{example}

\begin{example}{\it Grushin case.} In the state space $\re^{2}$, consider the matrix
\[
B(x)=\begin{pmatrix}
\!\!&1&0\!\\
\!\!&0&x_1\!\\
\end{pmatrix}.
\]
This matrix fulfills assumptions $(H4)$-$(H6)$. In particular, Corollary~\ref{corollary} deals with MFG with Grushin dynamics with unbounded coefficients which was not enconpassed among the cases coped in \cite[Theorem 1.1]{MMMT}.
\end{example}

\begin{remark}
Note that our result holds also for more general structures that do not satisfy assumptions of \cite[Theorem 18.2.1]{BLU}; indeed, for instance, the case in \cite[Example 18.1.4]{BLU} satisfies our assumptions even if is not a $H$-type group.
\end{remark}

\begin{remark}
Let us finally observe that the results in Corollary~\ref{corollary} can be further generalized with some slight modifications of the previous arguments. For instance, one can weaken assumption~$(H5)$ requiring only
\begin{itemize}
\item[(H5')] for $j\in\{1,\dots,m\}$, $h_{ij}$ has an at most linear growth with respect to $(x_1,\dots,x_j)$ uniformly in $x$.
\end{itemize}
In this case the proof is similar to the one of Corollary~\ref{corollary}. We only give some details on the proof that the any solution to~\eqref{1802:rmk2_i} is bounded and Lipschitz continuous. Indeed equation~\eqref{1802:rmk2_i} reads as the system~formed by equations~\eqref{OS_gen1} with $j\in\{1,\dots,n\}$. We proceed iteratively on~$j$. For $j=1$, equation~\eqref{OS_gen1} has a right hand side which has an at most linear growth only in $\gamma_1$ uniformly in $\gamma$. We deduce that $\gamma_1$ is bounded and Lipschitz continuous. Afterwards, by these properties of $\gamma_1$, we deduce that the right hand side of equation~\eqref{OS_gen1} with $j=2$  has an at most linear growth only in $\gamma_2$; hence also $\gamma_2$ is bounded and Lipschitz continuous. Iteratively, we get that the whole $\gamma$ is bounded and Lipschitz continuous.\\
We here give an example of a  $4 \times 3$ matrix which satisfies (H5') but not (H5):
\[
B(x)=\begin{pmatrix}
\!\!&1&0&0\!\\
\!\!&h_{21}(x_1)&x_1&0\!\\
\!\!&x_1+x_2&x_1+x_2&x_1+x_2\!\\
\!\!&x_1+x_2+x_3&x_1+x_2+x_3&x_1+x_2+x_3\!\\
\end{pmatrix}
\]
where the function $h_{21}$ is bounded and with at most a linear growth at infinity. Indeed
assumption $(H5)$ requires also that the terms $h_{22}$ and $h_{3i}$, $i=1,2,3$ are bounded.
\end{remark}

\subsection{Hamiltonian of the form $H(x,p)=|pB(x)|^\gamma$ with $\gamma\in[1,2]$}
Our results can be extended to any Hamiltonian of the form $H(x,p)=|pB(x)|^\gamma$ with $\gamma\in[1,2]$. 
We can write explicitely the Hamiltonian
$$
H(x,p):=((p_1-x_2p_3)^2+(p_2+x_1p_3)^2))^{\gamma/2}$$
and the drift term $\partial_pH(x,p)$ in the continuity equation is
\begin{eqnarray*}\label{Hp}
&&\partial_pH(x,p)=\gamma (|pB(x)|^2)^{\gamma/2-1}pB(x)B(x)^T\\
&&=\gamma \frac{1}{(|pB(x)|^2)^{1-\gamma/2}}(p_1-x_2p_3, p_2+x_1p_3, -p_1x_2+p_2x_1+p_3(x_1^2+x_2^2)).
\end{eqnarray*}
Since $\gamma\in[1,2]$ this term has still a sublinear growth.\\
In this case in the cost functional for the associated control problem we have to replace the term 
$\frac{1}{2} |\alpha|^2$ with $\frac{1}{2} |\alpha|^{\gamma'}$,
where $\frac{1}{\gamma}+ \frac{1}{\gamma'}=1$.\\ 
The results on the associated optimal control, found in Section \ref{OC} for $\gamma=2$, hold also in this case. In particular
from the optimality conditions, the optimal trajectories $x^*(s)$ associated to the optimal control problem given in \eqref{tage} for $\gamma=2$ satisfy:
\begin{equation}\label{systemgamma}
x'= \partial_pH(x,Du)=\gamma \frac{1}{|pB(x)|^{2-\gamma}}(u_{x_1}-x_2u_{x_3}, u_{x_2}+x_1u_{x_3}, -u_{x_1}x_2+u_{x_2}x_1+u_{x_3}(x_1^2+x_2^2)).
\end{equation}
A key point to obtain our main result is the optimal synthesis proved in Lemma \ref{lemma:OSeps} for the case $\gamma=2$, in particular the Lipschitz continuity of the solution of system \eqref{1802:rmk2_i}.
In this case, since $Du$ is still bounded, the first two components of the right hand side of \eqref{systemgamma} has sublinear growth. Hence we can argue as in point a) of the proof of Lemma \ref{lemma:OSeps} to obtain that there exists a constant $C$ such that $\xi(s):=(x_1(s),x_2(s))$ satisfies 
$|\xi|'\leq C(|\xi|^{\beta}+1)$ with $\beta<1$. Then $\xi(s):=(x_1(s),x_2(s))$ is bounded and from the third component of \eqref{systemgamma} also 
$x_3(s)$ in bounded, then the curve $x(s)$ is Lipschitz continuous.
Since the drift term in the continuity equation has a sublinear behaviour we can repeat the arguments of Section \ref{sect:c_illi} to get also in this case the main result.

\vskip .3truecm
\noindent{\bf Acknowledgments.} 
The first and the second authors are members of GNAMPA-INdAM and were partially supported also by the research project of the University of Padova ``Mean-Field Games and Nonlinear PDEs'', by the Fondazione CaRiPaRo Project ``Nonlinear Partial Differential Equations: Asymptotic Problems and Mean-Field Games'' and by KAUST project OSR-2017-CRG6-3452.01. The third author has been partially funded by the ANR project ANR-16-CE40-0015-01.

{\small{
 }

\end{document}